\documentclass[11pt]{amsart}

\usepackage{amsmath}
\usepackage{amsfonts}
\usepackage{amssymb}
\usepackage{amsthm}
\usepackage{graphics}
\usepackage{amscd}
\usepackage{graphicx,epsfig}
\usepackage{color}
\usepackage[all]{xy}
\usepackage{url}

\renewcommand{\epsilon}{\varepsilon}

\newcommand{\boF}{\mathcal{F}}

\newcommand{\boM}{\mathcal{M}}
\newcommand{\boH}{\mathcal{H}}
\newcommand{\boA}{\mathcal{A}}

\newcommand{\boO}{\mathcal{O}}
\newcommand{\boL}{\mathcal{L}}
\newcommand{\boS}{\mathcal{S}}

\newcommand{\boU}{\mathcal{U}}
\newcommand{\boZ}{\mathcal{Z}}

\newcommand{\LL}{\mathbf{L}}
\newcommand{\MM}{\mathbf{M}}
\newcommand{\bfd}{\mathbf{d}}
\newcommand{\bff}{\mathbf{f}}
\newcommand{\bfm}{\mathbf{m}}

\newcommand{\R}{\mathbb{R}}

\newcommand{\Z}{\mathbb{Z}}

\renewcommand{\H}{\mathbb{H}}
\renewcommand{\S}{\mathbb{S}}
\newcommand{\N}{\mathbb{N}}

\newcommand{\eps}{\varepsilon}

\newcommand{\dd}{\mathrm{d}}

\newcommand{\Ome}{\Omega}

\newcommand{\wtilde}{\widetilde}

\newtheorem{thm}{Theorem}
\newtheorem{defn}[thm]{Definition}

\newtheorem{prop}[thm]{Proposition}
\newtheorem{lem}[thm]{Lemma}

\newcommand{\barre}[1]{\overline{#1}}
\renewcommand{\phi}{\varphi}

\newtheorem*{thm*}{Theorem}

\newtheorem{claim}{Claim}
\newtheorem*{claim*}{Claim}

\theoremstyle{remark}

\newtheorem{remarq}{Remark}
\newtheorem*{rem*}{Remark}

\newcounter{remark}

\newcounter{case}

\newcounter{construction}

\newcounter{fact}

\newcounter{step}

\DeclareMathOperator{\Ric}{Ric}

\theoremstyle{plain}
\newtheorem*{thA}{Theorem A}
\newtheorem*{thB}{Theorem B}
\newtheorem*{thC}{Theorem C}

\title{Minimal hypersurfaces of least area}

\author{Laurent Mazet}
\address{Universit\'e Paris-Est, LAMA (UMR 8050), UPEC, UPEM, CNRS, 61, avenue
du G\'en\'eral de Gaulle, F-94010 Cr\'eteil cedex, France}
\email{laurent.mazet@math.cnrs.fr}

\author{Harold Rosenberg} 
\address{Instituto Nacional de Matematica Pura e Aplicada (IMPA) Estrada Dona
Castorina 110, 22460-320, Rio de Janeiro-RJ, Brazil}
\email{rosen@impa.br}
\thanks{The authors were partially supported by the ANR-11-IS01-0002 grant.}


\begin{document}

\maketitle

\begin{abstract}
In this paper, we study closed embedded minimal hypersurfaces in a Riemannian
$(n+1)$-manifold ($2\le n\le 6$) that minimize area among such hypersurfaces. We show they
exist and arise either by minimization techniques or by min-max methods : they have index
at most $1$. We apply this to obtain a lower area bound for such minimal surfaces in some
hyperbolic $3$-manifolds.
\end{abstract}


\section{Introduction}

A classical result in minimal hypersurfaces theory is that, in $\S^{n+1}$ with the round
metric, the totally geodesic equatorial $\S^n$ has least area among minimal hypersurfaces
in $\S^{n+1}$. Actually it is a consequence of the monotonicity formula for minimal
hypersurfaces in $\R^{n+2}$. Another consequence of the monotonicity formula in a general
closed Riemannian manifold $M$ is that any minimal hypersurface has area at least some
positive constant depending on $M$. So one can ask to precise this constant or
to find a minimal hypersurface of least area among minimal hypersurfaces in $M$.

One way to understand this question is to look at how minimal hypersurfaces can be
constructed as critical points of the area functional in a closed Riemannian
$(n+1)$-manifold $M$. If $S$ is some closed hypersurface in $M$ non vanishing in homology,
geometric measure theory \cite{Fed} tells us that the area can be minimized in the
homology class of $S$ to produce a closed embedded minimal hypersurface $\Sigma$ in $M$
which minimizes the area. Actually, $\Sigma$ is a smooth hypersurface outside some
singular subset of Hausdorff dimension less than $n-7$. This approach produces minimal
hypersurfaces that are stable, \textit{i.e.} the Jacobi operator on $\Sigma$ has index
$0$.

If the homology group $H_n(M)$ vanishes, for example $M=\S^{n+1}$, the above idea can not
be applied. Almgren and Pitts
\cite{Alm,Pit2} then developed a min-max approach to construct minimal hypersurfaces in such
a manifold $M$. They prove that the fundamental class $[M]\in H_{n+1}(M)$ is associated to
a particular positive number $W_M$ called the width of the manifold. Then this number is
realized as the area of some particular minimal hypersurface (maybe with multiplicities);
this minimal hypersurface is called a min-max hypersurface associated to the
fundamental class $[M]$. Pitts proved the result when $2\le n\le 5$, it was extended by
Schoen and Simon~\cite{ScSi} later to higher values of $n$. Here also, the minimal
hypersurface may have a singular subset of Hausdorff dimension less than $n-7$. As a
consequence, there always exists a smooth minimal hypersurface in $M$ if $2\le n\le 6$.
This min-max approach works with one parameter families of hypersurfaces called sweep-outs,
so the min-max
hypersurface is expected to have index at most $1$. For example, all such min-max hypersurfaces in
the round $\S^{n+1}$ are the equatorial $\S^n$.

Coming back to the question of finding a minimal hypersurface of least area among minimal
hypersurfaces, the main result of our paper mainly says that such a hypersurface exists
and can be constructed by one of the above approaches. To be more precise, we take into account
the possible non orientability of hypersurfaces : let $\boO$ be the collection of all
smooth orientable connected closed embedded minimal hypersurfaces in $M$ and $\boU$ be the
collection
of the non orientable ones. If $2\le n\le 6$, we know that at least one of them is non
empty. We then define
$$
\boA_1(M)=\inf (\{|\Sigma|,\, \Sigma\in \boO\}\cup\{2|\Sigma|,\,\Sigma\in \boU\})
$$
where $|\cdot|$ denotes the area. The non orientable hypersurfaces are chosen to be
counted twice since, in several constructions, non orientable minimal hypersurfaces
appear with multiplicity $2$. So our main theorem can be stated as follows.

\begin{thA}
Let $M$ be an oriented closed Riemannian $(n+1)$-manifold ($2\le n\le 6$). Then
$\boA_1(M)$ is equal to one of the following possibilities.
\begin{enumerate}
\item $|\Sigma|$ where $\Sigma\in \boO$ is a min-max hypersurface of $M$ associated to the
fundamental class of $H_{n+1}(M)$ and has index $1$.

\item $|\Sigma|$ where $\Sigma\in\boO$ is stable.
\item $2|\Sigma|$ where $\Sigma\in \boU$ is stable and its orientable $2$-sheeted
cover has index $0$ or $1$; if the index is $1$, $2|\Sigma|=W_M$.
\end{enumerate}

Moreover, if $\Sigma\in\boO$ satisfies $|\Sigma|=\boA_1(M)$, then $\Sigma$ is of type $1$ or
$2$ and if $\Sigma\in\boU$ satisfies $2|\Sigma|=\boA_1(M)$, then $\Sigma$ is of type $3$.
\end{thA}

So the theorem says that $\boA_1(M)$ is realized, moreover it characterizes all minimal
hypersurfaces that realize $\boA_1(M)$. Let us first notice that the restriction on the
dimension is the classical restriction about the regularity for minimal hypersurfaces in
high dimensions. The main property of the hypersurface $\Sigma$ is expressed in terms of the
index of its Jacobi operator: it is $0$ (stable case) or $1$. 

If $M$ has positive Ricci curvature, it is known that there is no stable orientable minimal
hypersurface. So, in that case, the above theorem is similar to the main result obtained
in \cite{Zho} where Zhou characterizes the min-max hypersurface in the positive Ricci
case. Actually the estimate on the index of the double cover in the non orientable case
does not appear in the work of Zhou. For the rest, the proof of our result is based on
similar ideas to the work of Zhou.

Of course, it would be interesting to say more about the hypersurface that appears in
Theorem~A, for example about its topology. In dimension $3$ ($n=2$), we are
able to give some improvements to our main results. In fact we prove that in the index $1$
case for type $1$ and $3$ surfaces, the genus of the surface $\Sigma$ can not be to small
and is controlled by the Heegaard genus of the ambient manifold $M$; this will be
Theorem~B. In \cite{MaNe3},
Marques and Neves look also for control on the genus of index~$1$ minimal surfaces. In
fact, finding upper bounds for the genus of min-max surfaces was first present in the work
of Smith \cite{Smi} about the existence of minimal $2$-spheres in Riemannian $3$-spheres and
has received major contributions by De~Lellis and Pellandini~\cite{DeLPe} and
Ketover~\cite{Ket}.

Actually, sometimes, index and genus can be combined to estimate the area of a minimal surface
(see \cite{MaNe3} for an example). So one consequence of our improvement is that we can
give a lower bound for the area of minimal surfaces in hyperbolic $3$-manifolds.

\begin{thC}
Let $M$ be a closed orientable hyperbolic $3$-manifold. If the Heegaard genus of $M$ is at
least $7$ then $\boA_1(M)\ge 2\pi$. In other words, any orientable minimal surface in $M$
has area at least $2\pi$ and any non orientable minimal surface has area at least $\pi$.
\end{thC}

Let us notice that, in the above result, if $M$ does not have any non orientable surface, we
need only assume that the Heegard genus is at least~$6$.

To prove Theorem~A, one idea would be to consider a minimizing sequence and use
some compactness result for minimal hypersurfaces to get some limit hypersurface. The
main default with this approach is that \textit{a priori} the eventual limit need not be a
smooth hypersurface. However, this minimization argument can be done among stable minimal
hypersurfaces to produce a stable minimal hypersurface with least area. So we can construct a
stable minimal hypersurface that realizes $\boA_\boS(M)$ where $\boA_\boS(M)$ is defined
as $\boA_1(M)$ but with an infimum computed only among stable minimal hypersurfaces. If
$\boA_1(M)=\boA_\boS(M)$, this almost gives the proof of the main theorem.

In fact, the proof of Theorem~A mainly consists in proving that
$\boA_1(M)=\min(W_M,\boA_\boS(M))$. So we need to understand minimal hypersurfaces
$\Sigma$ with
area less than $\boA_\boS(M)$. Actually, we prove that $\Sigma$ can be seen as a leaf of
maximal area of some sweep-out of the manifold $M$. As a consequence, this implies that the
area of the min-max hypersurface constructed by Pitts is less than the area of $\Sigma$ so
this min-max hypersurface has to realize $\boA_1(M)$. The proof of the existence of the
above sweep-out uses another point of view about min-max theory for minimal hypersurfaces
which is developed by Colding, De Lellis and Tasnady~\cite{CoDeL,DeLTa}.

Actually the questions we look at in this paper can be generalized. Consider the space
$\boM=\boO\cup \boU$ of closed embedded minimal hypersurfaces on a manifold $M$. Let
$\boA:\boM\to \R^+$ be the area function. What are the properties of $\boA$? Is $\boA$
always unbounded ? In this paper
we discussed $\boA_1(M)$, the minimum of $\boA$. In general the values of $\boA$ are
difficult to understand. For example, when $M$ is the standard $3$-sphere, we know that
$\boA_1(M)= 4\pi$. The next value of $\boA$ is $2\pi^2$, the area of the Clifford torus.
This is very difficult to prove and is an important part of the solution of the Willmore
conjecture by Marques and Neves \cite{MaNe}. Let us also notice that there is a gap in the
values of $\boA$ after $2\pi^2$. Then one has the Lawson examples that are genus $g$
surfaces in $\boM$ whose areas converge to $8\pi$ as $g\to \infty$. One can see these
surfaces (as $g\to\infty$) as desingularizing two orthogonal geodesic $2$-spheres along
their intersection.

Actually we do not believe $8\pi$ can be realized as the area of a minimal surface.

By desingularizing $k$ geodesic $2$-spheres meeting along a common geodesic at equal
angles, one obtains surfaces in $\boM$ whose areas converge to $4\pi k$.

For any $M$ of dimension $3$, one can consider the surfaces in $\boM$ of genus at most $g$
(there may not be any) and try to calculate the minimum $\boA^g(M)$ of $\boA$ on these
surfaces. $\boA^g(M)$ is realized (provided some genus $g$ surface exists in $M$). The
behaviour of $\boA^g(M)$ would be interesting to understand.

As an example, what are these quantities for the space $M=\S^2\times\S^1$, $\S^2$ the unit
$2$-sphere and $\S^1$ the circle of length $\ell$\,?

We also notice that Theorem~A does not solve the following question: if $(\Sigma_n)_n$ is a
sequence of minimal hypersurfaces whose areas converge to $\boA_1(M)$, do we have convergence
of $(\Sigma_n)_n$ to one of the smooth hypersurfaces of Theorem~A?

This article is organized as follow. In Section~\ref{sec:minimal}, we recall some classical
definitions about the index of minimal hypersurfaces. In Section~\ref{sec:minmax}, we give a
quick presentation of the min-max theories of Colding, De Lellis and Tasnady (the
continuous setting) and of Almgren and Pitts (the discrete setting).

Section~\ref{sec:minstable} is devoted to the minimization among stable hypersurfaces, we
define $\boA_\boS(M)$ and prove that it is realized. In Section~\ref{sec:sweepout}, we
construct the sweep-out associated to a minimal hypersurface with area less than
$\boA_\boS(M)$. Finally the proof of the main theorem is given in Section~\ref{sec:mainthm}.

From Section~\ref{sec:mainthm3}, we look at the dimension $3$ case ; in
Section~\ref{sec:mainthm3}, we improve Theorem~A to obtain some control of the
topology of the surface. This result is then applied in Section~\ref{sec:hyperbolic} to give
a lower bound for the area of minimal surfaces in hyperbolic $3$-manifolds.

Our work is strongly influenced by the paper of Marques and Neves \cite{MaNe3}. Indeed, at
a recent meeting, when we told Fernando C. Marques about our work, he returned the next
day with the ideas we had used to prove $\boA_1(M)$ is realized.

\section{Minimal hypersurfaces}\label{sec:minimal}

In this section, we give some definitions and recall some basic facts about minimal
hypersurfaces.

In this paper, we look at hypersurfaces $\Sigma$ in a certain Riemannian $(n+1)$-manifold
$M$. All along the paper, $M$ will be orientable. If it is not precised, all hypersurfaces
are assumed to be embedded.


\subsection{Minimal hypersurfaces}

Minimal hypersurfaces in $M$ are those with vanishing mean curvature vector, they appear
as critical points of the area functional for hypersurfaces. 

In the following, we will denote by $\boO$ the collection of all orientable minimal
hypersurfaces and by $\boU$ the collection of all non orientable ones.

As in the introduction, we define
$$
\boA_1(M)=\inf(\{|\Sigma|,\Sigma\in \boO\}\cup\{2|\Sigma|,\Sigma\in\boU\})
$$


\subsection{The stability operator}

Minimal hypersurfaces are critical points of the area functional on hypersurfaces. The
study of the second derivative of the area functional on such a critical point is given by
the stability operator.

Let $\Sigma$ be a minimal hypersurface in an orientable Riemannian $(n+1)$-manifold $M$.
The stability operator is a quadratic differential form acting on sections of the normal
bundle $N\Sigma$ to $\Sigma$. If $\xi\in\Gamma(N\Sigma)$ is such a section, we have
$$
Q_\Sigma(\xi,\xi)=\int_\Sigma \|\nabla^\perp\xi\|^2- \Ric_M(\xi,\xi)-\|A\|^2\|\xi\|^2
d\text{vol}_\Sigma
$$
where $\nabla^{\perp}$ is the normal connection on $N\Sigma$ coming from the Levi-Civita
connection on $M$, $\Ric_M$ is the Ricci curvature tensor on $M$ and $\|A\|$ is the norm
of the second fundamental form on $\Sigma$.

A minimal hypersurface is called stable if $Q$ is non-negative. This means that $\Sigma$ is
a minimum at order $2$ for the area functional. The index of $\Sigma$ is the maximal
dimension of linear subspaces $E$ of $\Gamma(N\Sigma)$ such that $Q$ is negative definite
on $E$.

If $\Sigma$ is $2$-sided, \textit{i.e.} $N\Sigma$ is a trivial line bundle, there is a
unit normal vector field $\nu$ along $\Sigma$ so any section $\xi$ can be written as
$\xi=u\nu$ where $u$ is a function. Thus, the stability operator becomes an operator on
functions
$$
Q_\Sigma(u,u)=\int_\Sigma \|\nabla u\|^2- (\Ric_M(\nu,\nu)-\|A\|^2)u^2
d\text{vol}_\Sigma=-\int_\Sigma u\boL_\Sigma u\, d\text{vol}_\Sigma
$$ 
where $\boL_\Sigma u=\Delta u+(\Ric_M(\nu,\nu)+\|A\|^2)u$ is called the Jacobi operator on
$\Sigma$. 

If $\Sigma$ is a closed minimal hypersurface, $-\boL_\Sigma$ has a discrete spectrum
$\lambda_1<\lambda_2\le\cdots$. The index of $\Sigma$ is then the number of negative
eigenvalues of $-\boL_\Sigma$.

If $\Sigma$ is orientable, $\Sigma$ is $2$-sided since $M$ is orientable and the above
description applies. If $\Sigma$ is non orientable, $\Sigma$ is not $2$-sided but we can
consider $\pi: \wtilde\Sigma\to \Sigma$ the orientable double cover of $\Sigma$. The map
$\pi$ defines a minimal immersion of $\wtilde\Sigma$ in $M$ which is $2$-sided so the
Jacobi operator $\boL_{\wtilde\Sigma}$ is defined. The covering map $\pi$ comes with a
unique non trivial deck transformation $\sigma$ which is an involution. If $\nu$ is a unit
normal vector field along $\wtilde\Sigma$ we have $\nu\circ \sigma=-\nu$. So sections of
$N\Sigma$ correspond to $\sigma$-odd functions on $\wtilde\Sigma$ and $\sigma$-even
functions on $\wtilde\Sigma$ correspond to functions on $\Sigma$. We also notice that,
for a function $u$ on $\wtilde\Sigma$,
$\boL_{\wtilde\Sigma}(u\circ\sigma)=(\boL_{\wtilde\Sigma}u)\circ\sigma$. Thus the
hypersurface $\Sigma$ is stable if and only if $Q_{\wtilde\Sigma}$ is non negative on
$\sigma$-odd functions. As another consequence, if $\Sigma$ is stable and $u$ is an
eigenfunction of $-\boL_{\wtilde\Sigma}$ with a negative eigenvalue, $u$ is
$\sigma$-even.


\section{Preliminaries about min-max theory}\label{sec:minmax}

In this paper, we will use several times the min-max approach to construct minimal
hypersurfaces. There are two major settings for the min-max theory: the discrete setting
which is due to Almgren and Pitts \cite{Pit2} and the continuous setting due to Colding,
De Lellis \cite{CoDeL} and De Lellis, Tasnady \cite{DeLTa}. Both settings have their own
interest, the continuous setting is easier to consider for some geometric considerations
and the discrete setting is more linked to the topology of the ambient space.

Good introductions to both settings can be found in several papers (see \cite{CoDeL,DeLTa,MaNe,Zho}).
So here, we only summarize facts that we will really use.

Let $M$ be a compact Riemannian $(n+1)$-manifold with or without boundary. $\boH^k$ will
denote the $k$-dimensional Hausdorff measure and, when $\Sigma$ is a $n$-dimensional
submanifold, we use the following notation $|\Sigma|=\boH^n(\Sigma)$ and we say that
$|\Sigma|$ is the area of $\Sigma$ even if it has dimension larger than $2$. 
If $\Sigma$ is an immersed hypersurface, we also use $|\Sigma|$ to compute its volume
which could be different from the $\boH^n$-measure of its image in $M$.


\subsection{The continuous setting}

Let us recall some definitions and results from the papers of De Lellis and Tasnady
\cite{DeLTa} and Zhou \cite{Zho}. First let us define what kind of family of hypersurfaces
we will consider.

\begin{defn}
A family $\{\Gamma_t\}_{t\in [a,b]}$ of closed subsets of $M$ with finite $\boH^n$-measure
is called a \emph{generalized smooth family} if
\begin{itemize}
\item[(s1)] For each $t$ there is a finite set $P_t\subset M$ such that $\Gamma_t\setminus
P_t$ is either a smooth hypersurface in $M\setminus P_t$ or the empty set;
\item[(s2)] $\boH^n(\Gamma_t)$ depends continuously on $t$ and $t\mapsto \Gamma_t$ is
continuous in the Hausdorff sense;
\item[(s3)] on any $U\subset\subset M\setminus P_{t_0}$, $\Gamma_t\xrightarrow[]{t\to t_0}
\Gamma_{t_0}$ smoothly in $U$.
\end{itemize}
\end{defn}

Now let us define the continuous sweep-outs for ambient manifolds with or without boundary.

\begin{defn}
Let $M$ be a closed manifold. A generalized smooth family $\{\Gamma_t\}_{t\in[a,b]}$ is a
\emph{continuous sweep-out} of $M$ if there exists a family $\{\Ome_t\}_{t\in[a,b]}$ of
open subsets of $M$ such that
\begin{itemize}
\item[(sw1)] $(\Gamma_t\setminus \partial\Ome_t)\subset P_t$ for any $t$;
\item[(sw2)] $\boH^{n+1}(\Ome_t\triangle \Ome_s)\to 0$ as $t\to s$ (where $\triangle$
denotes the symmetric difference of subsets).
\item[(sw3)] $\Ome_a=\emptyset$ and $\Ome_b=M$;
\end{itemize}
\end{defn}

\begin{defn}
Let $M$ be a compact manifold with non empty boundary. A generalized smooth family
$\{\Gamma_t\}_{t\in[a,b]}$ is a \emph{continuous sweep-out} of $M$ if there exists a
family $\{\Ome_t\}_{t\in[a,b]}$ of open subsets of $M$ satisfying \emph{(sw2)} and
\begin{itemize}
\item[(sw0')] $\partial M\subset \Ome_t$ for $t>a$;
\item[(sw1')] $(\Gamma_t\setminus \partial_*\Ome_t)\subset P_t$ for any $t>a$ where
$\partial_*\Ome_t=\partial\Ome_t\setminus \partial M$;
\item[(sw3')] $\Ome_a=\emptyset$, $\Ome_b=M$ and there are $\eps>0$ and a smooth function
$w:[0,\eps]\times\partial M\to \R$ with $w(0,p)=0$ and $\partial_tw(0,p)>0$ such that 
$$
\Gamma_{a+t}=\{\exp_p(w(t,p)\nu(p)), p\in \partial M\}
$$
for $t\in[0,\eps]$ and $\nu$ the inward unit normal to $\partial M$.
\end{itemize}
\end{defn}

For a continous sweep-out as above $\{\Gamma_t\}_{t\in[a,b]}$, we define the quantity
$\LL({\Gamma_t})=\max_{t\in[a,b]}|\Gamma_t|$.

Two continuous sweep-outs $\{\Gamma_t^1\}_{t\in[a,b]}$ and $\{\Gamma_t^2\}_{t\in[a,b]}$ are said to
be homotopic if, informally, they can be continuously deformed one to the other (the precise
definitions are Definition 0.6 in \cite{DeLTa} and Definition 2.5 in \cite{Zho}). Then a
family $\Lambda$ of sweep-outs is called homotopically closed if it contains the homotopy
class of each of its elements. For such a family $\Lambda$, we can define the width
associated to $\Lambda$ as
$$
W(\Lambda)= \inf_{\{\Gamma_t\}\in\Lambda} \LL(\{\Gamma_t\})
$$

We notice that when $M$ has no boundary, $W(\Lambda)>0$ for any $\Lambda$ (see Proposition
0.5 in \cite{DeLTa}).

If $\Lambda$ is a homotopically closed family of sweep-outs and the sequence
$(\{\Gamma_t^k\}_t)_{k\in\N}$ of sweep-outs is such that
$\LL(\{\Gamma_t^k\}_t)\xrightarrow[k\to \infty]{}W(\Lambda)$, a min-max sequence is a
sequence $(\Gamma_{t_k}^k)$ (or a subsequence of this sequence) such that
$|\Gamma_{t_k}^k|\xrightarrow[k\to \infty]{}W(\Lambda)$. The main existence-result about
the min-max theory in this setting is (see Theorem 0.7 \cite{DeLTa} and Theorem 2.7 \cite{Zho})

\begin{thm}[De Lellis, Tasnady \cite{DeLTa}, Zhou \cite{Zho}]\label{thm:contminmax}
Let $M$ be a compact Riemannian $(n+1)$-manifold ($2\le n\le 6$). Let $\Lambda$ be a
homotopically closed family of continuous sweep-outs of $M$. If $M$ has no
boundary, there is a min-max sequence that converges (in the varifold sense) to an
integral varifold whose support is a finite collection of embedded connected disjoint minimal
hypersurfaces of $M$. As a consequence 
$$ W(\Lambda)=\sum_{i=1}^pn_i|S_i|$$
where $\cup_{i=1}^p S_i$ is the support of the limit varifold.

If $M$ has boundary, the same result is true if we assume that the mean curvature vector
of $\partial M$ does not vanish and points into $M$ and $W(\Lambda)>|\partial M|$.
\end{thm}

We refer to \cite{Sim} for the definition of the convergence in varifold sense.

\begin{remarq}
One consequence of this result that we will use is that if we have some continuous sweep-out
$\{\Gamma_t\}_t$ of $M$ ($\partial M=\emptyset$) then there is some connected minimal
hypersurface $S$
in $M$ with $|S|\le \LL(\{\Gamma_t\})$. 
\end{remarq}


\subsection{The discrete setting}

Here we recall some aspects of the Almgren-Pitts min-max theory which deals with discrete
families of elements of $\boZ_n(M)$ \textit{i.e} integral rectifiable $n$-currents in $M$
with no boundary. For definitions about currents, we refer to \cite{Fed,Sim}.

If $I=[0,1]$, we first introduce some cell complex structure on $I$ and $I^2$.
\begin{defn}
Let $j$ be an integer, we define $I(1,j)$ to be the cell complex of $I$, whose $0$-cells
are the points $[\frac i{3^j}]$ for $i=0,\dots, 3^j$ and the $1$-cells are the intervals
$[\frac i{3^j},\frac{i+1}{3^j}]$ for $i=0,\dots, 3^j-1$.

We also define a cell complex $I(2,j)$ on $I^2$ by $I(2,j)=I(1,j)\otimes I(1,j)$.
Similarly $I(m,j)$ can be defined on $I^m$.
\end{defn}

Let us introduce some notations about these cell complexes
\begin{itemize}
\item $I_0(1,j)$ denotes the set of the boundary $0$-cells $\{[0],[1]\}$.
\item $I(m,j)_0$ denotes the set of $0$-cells of $I(m,j)$.
\item The distance between two elements of $I(m,j)_0$ is 
$$
\bfd: I(m,j)_0\times I(m,j)_0\to \N\ ;\ (x,y)\mapsto 3^j\sum_{i=1}^m|x_i-y_i|
$$
\item The projection map $n(i,j) : I(m,i)_0\to I(m,j)_0$ is defined such that $n(i,j)(x)$
is the unique element in $I(m,j)_0$ such that 
$$\bfd(x,n(i,j)(x))=\inf\{\bfd(x,y), y\in I(1,j)_0\}.$$
\end{itemize}

We are going to look at maps $\phi : I(m,j)_0\to \boZ_n(M)$. For such a map $\phi$, we
define its fineness by
$$
\bff(\phi)=\sup\left\{\frac{\MM(\phi(x)-\phi(y))}{\bfd(x,y)},x,y\in I(m,j)_0 \text{ and
}x\neq y\right\}
$$
where $\MM$ denotes the mass of a current.

When we write $\phi :I(1,j)_0\to (\boZ_n(M),\{0\})$, we mean $\phi(I(1,j)_0)\subset
\boZ_n(M)$ and $\phi(I_0(1,j))=\{0\}$.

\begin{defn}
Let $\delta$ be a positive real number and $\phi_i : I(1,k_i)_0\to(\boZ_n(M),\{0\})$,
$i=1,2$. We say that $\phi_1$ and $\phi_2$ are \emph{$1$-homotopic in $(\boZ_n(M),\{0\})$
with fineness $\delta$} if there are $k_3\in\N$, $k_3\ge \max(k_1,k_2)$, and a map
$$
\psi:I(2,k_3)_0 \to \boZ_n(M)
$$
such that
\begin{itemize}
\item $\bff(\psi)\le \delta$;
\item $\psi(i-1,x)=\phi_i(n(k_3,k_i)(x))$ for all $x\in I(1,k_3)_0$;
\item $\psi(I(1,k_3)_0\times \{[0],[1]\})=0$.
\end{itemize}
\end{defn}

Let us now define the equivalent of generalized smooth family in the discrete setting.

\begin{defn}
A \emph{$(1,\MM)$-homotopy sequence of maps into $(\boZ_n(M),\{0\})$} is a sequence of
maps $\{\phi_i\}_{i\in\N}$,
$$
\phi_i:I(1,k_i)_0\to (\boZ_n(M),\{0\}),
$$
such that $\phi_i$ is $1$-homotopic to $\phi_{i+1}$ in $(\boZ_n(M),\{0\})$ with fineness
$\delta_i$ and
\begin{itemize}
\item $\lim_{i\to \infty}\delta_i=0$;
\item $\sup_i\{\MM(\phi_i(x)), x\in I(1,k_i)_0\}<+\infty$.
\end{itemize}
\end{defn}

As in the continuous setting, two $(1,\MM)$-homotopy sequences can be said to be homotopic
and this defines an equivalence relation (see Section 4.1 in \cite{Pit2} or Definition 4.4
in \cite{Zho}). The set of all equivalence  classes is denoted by
$\pi_1^\#(\boZ_n(M),\MM,\{0\})$. One of the main results of the Almgren-Pitts theory says
that $\pi_1^\#(\boZ_n(M),\MM,\{0\})$ is naturally isomorphic to the homology group
$H_{n+1}(M,\Z)$ (Theorem 4.6 in \cite{Pit2}, see also \cite{Alm}).

If $S=\{\phi_i\}_i$ is a $(1,\MM)$-homotopy sequence, we define the quantity 
$$
\LL(S)=\limsup_{i\to \infty}\max\{\MM(\phi_i(x)),x\in I(1,k_i)_0\}
$$
Now, if $\Pi\in\pi_1^\#(\boZ_n(M),\MM,\{0\})$ is an equivalence class, we can define the
width associated to $\Pi$ by
$$
W(\Pi)=\inf\{\LL(S),S\in \Pi\}
$$
The class that corresponds to the fundamental class in $H_{n+1}(M)$ by the
Almgren-Pitts isomorphism is denoted $\Pi_M$. If $S=\{\phi_i\}_i\in \Pi_M$, we say
that $S$ is a \emph{discrete sweep-out} of $M$. The width $W(\Pi_M)$ is denoted by $W_M$ and is
called the width of the manifold $M$. 

The theory tells us that there is $S\in\Pi_M$ such that $\LL(S)=W(\Pi_M)=W_M$. If
$S=\{\phi_i\}_i$, we then say that $\phi_{i_j}(x_j)$ is a min-max sequence ($x_j\in
I(1,k_{i_j})$\,) if $\MM(\phi_{i_j}(x_j))\to W_M$. The min-max theorem of the Almgren-Pitts
theory says the following (see \cite{Pit2} for $n\le 5$ and \cite{ScSi} for $n=6$).

\begin{thm}[Pitts \cite{Pit2}, Schoen-Simon \cite{ScSi}]\label{thm:discminmax}
Let $M$ be a closed Riemannian $(n+1)$-manifold ($2\le n\le 6$).
There is a $S=\{\phi_i\}_i\in\Pi_M$ with $\LL(S)=W_M$ and a min-max sequence
$\{\phi_{i_j}(x_j)\}_j$ that converges (in the varifold sense) to an integral varifold
whose support is a finite collection of embedded connected disjoint minimal hypersurfaces
of $M$. As
a consequence 
$$
W_M=\sum_{i=1}^pn_i|S_i|
$$ 
where $\cup_{i=1}^p S_i$ is the support of the limit varifold.
\end{thm}

A limit varifold as in the above theorem will be called a min-max varifold associated to
the fundamental class of $H_{n+1}(M)$ and by extension we say that its support is a min-max
minimal hypersurface associated to the fundamental class of $H_{n+1}(M)$.

\begin{remarq}
In \cite{Zho}, Zhou gives some precisions about the multiplicities that appear in the above theorem. He proved that if $S_i$ is a non orientable minimal hypersurface then its
multiplicity $n_i$ has to be even (Proposition 6.1 in \cite{Zho}).
\end{remarq}


\subsection{From continuous to discrete}

It is easy to construct a continuous sweep-out of a manifold : we can just look at the
level sets of a Morse function on the manifold $M$. The construction of a discrete
sweep-out is not as clear even if the Almgren-Pitts isomorphism tells us that they exist.

In order to make a link between continuous and discrete sweep-outs, we use the following
result (see Theorem 13.1 in \cite{MaNe} and Theorem 5.5 in \cite{Zho}).

\begin{thm}\label{thm:conttodisc}
Let $\{\Ome_t\}_{t\in[a,b]}$ be a family of open subsets of $M$ satisfying $(sw2)$, $(sw3)$ and 
\begin{itemize}
\item $\Phi(t)=\partial[\Ome_t]\in\boZ_n(M)$;
\item $\sup\,\{\MM(\Phi(t)),t\in[a,b]\}<+\infty$
\item $\bfm(\Phi, r)=\sup\{\|\Phi(t)\|B(p,r),p\in M\text{ and }t\in[a,b]\}\to 0$ as $r\to
0$ where $B(p,r)$ is the geodesic ball of $M$ of center $p$ and radius $r$ and
$\|\cdot\|$ denote the Radon measure on $M$ associated to a current. 
\end{itemize}

Then there is a $(1,\MM)$-homotopy sequence $S\in\Pi_M$ such that 
$$
\LL(S)\le \sup\,\{\MM(\Phi(t)),t\in[a,b]\}
$$
\end{thm}

\begin{remarq}\label{rmq:masscontrol}
Actually, the estimate on $\LL(S)$ comes from a much stronger property of the
construction. Let $\widetilde\Phi(t)=\Phi(a+t(b-a))$. The $(1,\MM)$ homotopy sequence
$S=\{\phi_i\}_{i\in\N}$ has the following property: there are sequences $\delta_i\to 0$
and $l_i\to \infty$ such that
\begin{equation}\label{eq:masscontrol}
\MM(\phi_i(x))\le \sup\{\MM (\widetilde\Phi(y)),\ x,y\in\alpha \textrm{ for
some }1\textrm{-cell }\alpha\in I(1,l_i)\}+\delta_i
\end{equation}

Another property of $S$ is that $\boF(\phi_i(x)-\widetilde\Phi(x))\le \delta_i$ for any
$x\in I(1,k_i)_0$ where $\boF$ is the flat norm on the space of currents and
$\phi_i:I(1,k_i)_0\to \boZ_n(M)$.
\end{remarq}

\begin{remarq}\label{rmq:noconcentr}
The hypothesis about $\bfm(\Phi,r)$ is a no concentration property of the family
$\{\Phi(t)\}_t$.
Actually, the above theorem is used to produce discrete sweep-outs from continuous ones.
This can be done since the hypotheses on $\bfm(\Phi,r)$ is satisfied if
$\Phi(t)=[\Gamma_t]$ where $\{\Gamma_t\}_t$ is a continuous sweep-out (see Proposition~5.1
in \cite{Zho}).
\end{remarq}


\section{Stable minimal hypersurfaces}\label{sec:minstable}

Among all minimal hypersurfaces, the stable ones play an important role since they appear
when certain minimization arguments are done among some class of hypersurfaces. As a
consequence, they are natural candidates for a minimal hypersurface with least area.

In this section, we study these minimization arguments and look at a stable minimal
hypersurface with least area.


\subsection{Non separating hypersurfaces}

We first look at hypersurfaces that do not separate $M$ in two connected components.

\begin{prop}\label{prop:miniorient}
Let $M$ be a compact Riemannian $(n+1)$-manifold ($2\le n\le 6$) with mean-convex boundary.
Let $\Sigma$ be an oriented hypersurface in $M$ that is not homologous to $0$. Then there
is a connected orientable stable minimal hypersurface $\Sigma'$ which is non-vanishing in homology
and such that $|\Sigma'|\le |\Sigma|$. Moreover, if $\Sigma$ is not a stable minimal
hypersurface then $|\Sigma'|<|\Sigma|$.
\end{prop}

Typically, this proposition will be applied to non separating hypersurfaces.

\begin{proof}
$\Sigma$ represents a non vanishing homology class in $H_n(M,\Z)$. In terms of geometric
measure theory, $\Sigma$ can be seen as an integral $n$-cycle $[\Sigma]$. We can then
minimize the mass among all integral cycles in the homology class
of $[\Sigma]$ (see 5.1.6 in \cite{Fed}). This produces an integral cycle homologous to
$[\Sigma]$ whose support is made of several smooth connected orientable stable minimal
hypersurfaces (see 5.4.15 in \cite{Fed} or \cite{Sim}). Since $[\Sigma]\neq 0$, there is
one connected component $\Sigma'$ of this support that does not vanish in homology, this
component satisfies the properties of the above proposition.

If $\Sigma$ is not a stable minimal hypersurface, it is clear that there are hypersurfaces
homologous to $\Sigma$ with area strictly less that $|\Sigma|$ ; so $|\Sigma'|<|\Sigma|$.
\end{proof}

Let us fix a definition.

\begin{defn}
Let $N$ and $M$ be two $n$-manifolds with boundary and $\phi:N\to M$ a smooth map.
$\phi$ is said to be \emph{locally invertible} if, for any point $p$ in $N$, $d\phi(p)$ is
invertible and there is a neighborhood $V$ of $p$ in $N$ such that $\phi$ is bijective from $V$
to $\phi(V)$ with smooth inverse map.
\end{defn}

This definition mainly deals with properties of the map at boundary points of $N$ : for
example, boundary points of $N$ are not necessarily sent to boundary points of $M$. The inclusion
$[-1,1]\hookrightarrow [-2,2]$ is locally invertible, the map $[-\pi,\pi]\to \S^1;
t\mapsto (\cos t,\sin t)$ is also locally invertible.

\begin{prop}\label{prop:openorient}
Let $\Sigma$ be a connected closed oriented non separating hypersurface in the interior of a
manifold $M$ with boundary.
Then there is a manifold $\widetilde M$ with boundary with two particular boundary components
$\Sigma_1$ and $\Sigma_2$ and a locally invertible smooth map $\phi :\widetilde M\to M$
such that $\phi:\widetilde M 
\setminus (\Sigma_1\cup \Sigma_2)\to M\setminus \Sigma$ is a diffeomorphism and for
$i=1,2$ $\phi:\Sigma_i\to \Sigma$ is a diffeomorphism.
\end{prop}

\begin{proof}
Let us fix some complete Riemannian metric on $M$. Let $\nu$ be some unit normal vector field
along $\Sigma$. The map $\Phi : \Sigma\times(-2\eps,2\eps) \to M; (p,t)\mapsto
\exp_p(t\nu(p))$ is a diffeomorphism on its image for small $\eps$. Let $\eps$ be so. Let
$M_\eps$ be $M\setminus \Phi(\Sigma\times[-\eps,\eps])$. We then define $\widetilde M$ as the
quotient of the disjoint union of $M_\eps$, $\Sigma\times[0,2\eps)$ and
$\Sigma\times(-2\eps,0]$ by the identifications $(p,t)\simeq \Phi(p,t)\in M_\eps$ for
$(p,t)$ in $\Sigma\times(-2\eps,-\eps)$ or $\Sigma\times(\eps,2\eps)$.

The map $\phi$ is then defined as the identity on $M_\eps$ and by $\Phi$ on
$\Sigma\times(-2\eps,0]$ and $\Sigma\times[0,2\eps)$. $\Sigma_1$ and $\Sigma_2$ are the
two copies of $\Sigma\times\{0\}$. The map $\phi$ clearly satisfies the expected
properties.
\end{proof}

In the following, we will say that $\widetilde M$ is obtained by opening $M$ along
$\Sigma$. In general, there will be a metric on $M$ so we always lift this metric to
$\widetilde M$ so that $\phi$ is a local isometry.


\subsection{Non orientable hypersurfaces}

In this section, we look at the area of non orientable minimal hypersurfaces in $M$.

\begin{prop}\label{prop:mininonorient}
Let $M$ be a closed orientable Riemannian $(n+1)$-manifold ($2\le n\le 6$) with
mean-convex boundary. Let $\Sigma$ be a
non-orientable hypersurface in $M$. Then there is a connected stable minimal
hypersurface $\Sigma'$ such that $|\Sigma'|\le | \Sigma|$. Moreover, if $\Sigma$ is
not a stable minimal hypersurface then $|\Sigma'|<|\Sigma|$.
\end{prop}

\begin{proof}
Since $M$ is orientable and $\Sigma$ is non-orientable, $\Sigma$ is not $2$-sided. Thus
$\Sigma$ represents a non vanishing element in $H_n(M,\Z/2\Z)$. In the geometric measure
theory setting, $\Sigma$ can also be seen as a flat chain modulo $2$ $[\Sigma]$ (see 4.2.26. in
\cite{Fed}). We can then minimize the mass among all flat chains modulo $2$ that
are homologous to $[\Sigma]$. We then get a flat chain $T$ modulo $2$ which is homologous to
$[\Sigma]$ and minimizes the mass. The support of $T$ is then made of a finite union of
disjoint smooth minimal hypersurfaces (the regularity theory for area-minimizing flat
chains modulo $2$ can be found in \cite{Mor2} Corollary 2.5 and Remark 1 ; it uses also Lemma
4.2 in \cite{Mor3}). Let $\Sigma'$ be one of these minimal hypersurfaces ; it could be
orientable or not but in both cases the area-minimizing property of $T$ implies that
$\Sigma'$ is stable.

If $\Sigma$ is not a stable minimal hypersurface, it is clear that there is a hypersurface
homologous to $\Sigma$ with area strictly less that $|\Sigma|$ ; so $|\Sigma'|<|\Sigma|$.
\end{proof}

As in the preceding section, we can open a manifold along a non orientable hypersurface.

\begin{prop}\label{prop:opennonorient}
Let $\Sigma$ be a connected closed non-orientable hypersurface in the interior of a manifold $M$ with
boundary. Then there is a manifold $\widetilde M$ with boundary with a particular boundary
component $\widetilde \Sigma$ and a locally invertible smooth map $\phi :\widetilde M\to
M$ such that $\phi:\wtilde M\setminus \wtilde \Sigma\to M\setminus \Sigma$ is a
diffeomorphism and $\phi:\widetilde \Sigma\to \Sigma$ is an orientable double cover of
$\Sigma$. 
\end{prop}

The proof is similar to the orientable case (Proposition~\ref{prop:openorient}, see also
Proposition~3.7 in \cite{Zho}).

\begin{proof}
As in the preceding subsection, we consider a complete metric on $M$. Let $\pi:\widetilde
\Sigma\to \Sigma$ be an orientable double cover of $\Sigma$ and let $\sigma$ be the non
trivial deck transformation of $\pi$. $\pi$ defines an immersion of
$\widetilde\Sigma$ to $M$ so we can consider $\nu$ a unit normal vector field along
$\widetilde \Sigma$ we have $\nu(\sigma(p))=-\nu(p)$. Let us consider the map
$\Phi:\widetilde\Sigma\times[0,2\eps)\to M:(p,t)\mapsto \exp_{\pi(p)}(t\nu(p))$. We can
chose $\eps$ so that $\Phi$ is a diffeomorphism from $\widetilde\Sigma\times(0,2\eps)$ to
a tubular $2\eps$-neighborhood of $\Sigma$ with $\Sigma$ removed. Let $M_\eps$ be
$M\setminus \Phi(\widetilde\Sigma\times[0,\eps])$. We then define $\widetilde M$ as the
quotient of the disjoint union of $M_\eps$ and $\widetilde \Sigma\times[0,2\eps)$ by the
identifications $(p,t)\simeq \Phi(p,t)\in M_\eps$ for $(p,t)$ in
$\widetilde\Sigma\times(\eps,2\eps)$.

The map $\phi$ is then defined as the identity on $M_\eps$ and by $\Phi$ on
$\widetilde\Sigma\times[0,2\eps)$. The map $\phi$ clearly satisfies the expected
properties.
\end{proof}

As an example, if $M$ is $\R P^3$ and $\Sigma$ is an equatorial $\R P^2$ then
$\widetilde M$ is a hemisphere of $\S^3$ bounded by an equator $\widetilde\Sigma$.


\subsection{The number $\boA_\boS$}

Let $M$ be a compact orientable Riemannian $(n+1)$-manifold with mean convex boundary
($2\le n\le 6$). If $M$ contains a non orientable
or non separating hypersurface then Propositions~\ref{prop:miniorient} and
\ref{prop:mininonorient} give the existence of some stable minimal hypersurface in $M$.
So let us assume that $M$ contains some stable minimal hypersurface, we define $\boO_\boS$ the
collection of connected orientable stable minimal hypersurfaces and $\boU_\boS$ the
collection of the connected non
orientable stable minimal hypersurfaces. We then define
$$
\boA_\boS(M)=\inf (\{|\Sigma|,\Sigma\in\boO_\boS\}\cup\{2|\Sigma|, \Sigma\in\boU_\boS\})
$$
This number is the "least area" of stable minimal hypersurfaces in $M$. If
$\boO_\boS\cup\boU_\boS=\emptyset$, $\boA_\boS(M)=+\infty$.

The main result of this section is that this number is realized.

\begin{prop}\label{prop:ministable}
The number $\boA_\boS(M)$ is realized if it is finite: either there exists $\Sigma\in
\boO_\boS$ such that $|\Sigma|=\boA_\boS(M)$ or $\Sigma\in \boU_\boS$ such that
$2|\Sigma|=\boA_\boS(M)$.
\end{prop}	

\begin{proof}
We can assume that there exists a sequence $(\Sigma_n)_{n\in\N}$ in $\boO_\boS$ (or in
$\boU_\boS$) such that $|\Sigma_n|\rightarrow \boA_\boS(M)$ (or $2|\Sigma_n| \to \boA_\boS(M)$).

If the sequence is in $\boO_\boS$, this is a sequence of stable minimal hypersurfaces
whose areas are uniformly bounded. Then we can apply a compactness result (see
\cite{ScSi} or Theorem 1.3 in \cite{DeLTa}) to prove that a subsequence converges in the graphical sense to an oriented
minimal hypersurface $\Sigma$ with multiplicity one or to a non-oriented minimal
hypersurface $\Sigma$ with multiplicity $2$. In the first case $|\Sigma|=\lim
|\Sigma_n|=\boA_\boS(M)$ and moreover $\Sigma$ is stable. In the second case, $\boA_\boS(M)\le
2|\Sigma|=\lim|\Sigma_n|=\boA_\boS(M)$, so $\boA_\boS(M)=2|\Sigma|$ and
Proposition~\ref{prop:mininonorient} implies that $\Sigma$ is stable.

If the sequence is in $\boU_\boS$, we can still apply the compactness result. Indeed, for any
ball $B$ of radius less than the injectivity radius of $M$, $\Sigma_n\cap B$ is orientable
and stable in the $2$-sided sense. In that case, $(\Sigma_n)_n$ converges to a non-oriented
stable minimal hypersurface with multiplicity $1$. We then have $\boA_\boS(M)\le 2|\Sigma|=\lim
2|\Sigma_n|=\boA_\boS(M)$, so $\boA_\boS(M)=2|\Sigma|$.
\end{proof}


\section{Minimal hypersurfaces with area less than $\boA_\boS(M)$}\label{sec:sweepout}

In this section, we study minimal hypersurfaces whose areas are less than $\boA_\boS(M)$.
Actually we are going to prove that such a minimal hypersurface can be seen as the leaf
of maximal area in some continuous sweep-out of the ambient manifold $M$.

Let $\Sigma$ be a minimal hypersurface in $M$. If $\Sigma$ is oriented and
$|\Sigma|<\boA_\boS(M)$, Proposition~\ref{prop:miniorient} tells us that $\Sigma$ separates
$M$ and it is unstable. If $\Sigma$ is non-orientable, Proposition~\ref{prop:mininonorient}
implies that $2|\Sigma|\ge \boA_\boS(M)$. So we are going to look at orientable, unstable,
separating minimal hypersurfaces.

\begin{prop}\label{prop:sweep}
Let $M$ be a closed orientable Riemannian $(n+1)$-manifold ($2\le n\le 6$). Let $\Sigma$
be a connected oriented minimal hypersurface
which is unstable and $|\Sigma|\le \boA_\boS(M)$. Then there is a continuous
sweep-out $\{\Sigma_t\}_{t\in[-1,1]}$ of $M$ such that $\Sigma_0=\Sigma$,
$\LL(\{\Sigma_t\})=|\Sigma|$ and, for any $\eps>0$, there is $\delta>0$ such that
$|\Sigma_t|\le |\Sigma|-\delta$ if $|t|\ge \eps$.

Moreover, if $u_1$ is the first eigenfunction of the Jacobi operator on $\Sigma$ and $\nu$ is
a unit normal vector field along $\Sigma$, the hypersurface $\Sigma_t$ is given by
$\Phi(\Sigma,t)$ for $t$ close to zero where
$$
\Phi:\Sigma\times\R\to M; (p,t)\mapsto \exp_p(tu_1(p)\nu(p))
$$
\end{prop}

The proof of Proposition~\ref{prop:sweep} consists in gluing together two continuous
sweep-outs given by the following proposition.

\begin{prop}\label{prop:onesidesweep}
Let $M$ be a compact Riemannian $(n+1)$-manifold ($2\le n\le 6$) with $\partial M=\Sigma$
connected, minimal and unstable. Moreover, we assume that $|\Sigma|\le \boA_\boS(M)$. Then
there is a continuous sweep-out $\{\Sigma_t\}_{t\in[0,1]}$ of $M$ such that
$\LL(\{\Sigma_t\})=|\Sigma|$ and, for any
$\eps>0$, there is $\delta>0$ such that $|\Sigma_t|\le |\Sigma|-\delta$ if $t\ge \eps$.

Moreover, if $u_1$ is the first eigenfunction of the Jacobi operator on $\Sigma$ and $\nu$ is
the inward unit normal vector field along $\Sigma$, the hypersurface $\Sigma_t$ is given by
$\Phi(\Sigma,t)$ for $t$ close to zero where
$$
\Phi:\Sigma\times[0,\eps]\to M; (p,t)\mapsto \exp_p(tu_1(p)\nu(p))
$$
\end{prop}

\begin{proof}
Since $\Sigma$ is unstable, the first eigenvalue $\lambda_1$ associated to $u_1$ is
negative. $u_1$ is a positive function. For $\eps>0$ small enough,
the map $\Phi: \Sigma\times[0,\eps]\to M; (p,t)\mapsto \exp_p(tu_1(p)\nu(p))$ is well
defined.

We then define $\Sigma_t=\Phi(\Sigma,t)$ and $M_t=M\setminus \Phi(\Sigma\times[0,t))$. If
$\eps$ is chosen small enough, the family $\{\Sigma_t\}_{t\in[0,\eps]}$ defines a foliation of a
neighborhood of $\Sigma$ and satisfies the property (sw3'). All the leaves $\Sigma_t$
($t>0$) have non
vanishing mean curvature vector pointing towards $M_t$. Also
$|\Sigma_t|$ decreases for $t$ close to $0$ and $|\Sigma_\eps|\le |\Sigma|-\delta$ for
some $\delta>0$. So in
order to construct the sweep-out announced in the proposition, it is sufficient to
construct a sweep-out $\{\Sigma_t\}_{t\in[\eps,1]}$ of $M_\eps$ such that
$\LL(\{\Sigma_t\}_{t\in[\eps,1]})\le |\Sigma|-\delta/ 2$ : indeed, we can glue such a
sweep-out with the foliation $\{\Sigma_t\}_{t\in[0,\eps]}$ to produce the continuous
sweep-out of $M$.

So let us assume by contradiction that any continuous sweep-out
$\{\Sigma_t\}_{t\in[\eps,1]}$ of
$M_\eps$ satisfies $\LL(\{\Sigma_t\}_{t\in[\eps,1]})\ge |\Sigma|-\delta/ 2\ge
|\Sigma_\eps|+\delta/2$. Then the min-max theorem for manifolds with boundary
(Theorem~\ref{thm:contminmax} or Theorem 2.7 in \cite{Zho}) implies the existence of a connected
minimal hypersurface $S$ in $M_\eps$. Let us now look at properties of this hypersurface
$S$.


\begin{claim}\label{claim:1}
The hypersurface $S$ is orientable
\end{claim}

If $S$ is not orientable, we can consider the manifold $\widetilde M_\eps$
constructed by opening $M_\eps$ along $S$ by Proposition~\ref{prop:opennonorient} with a
map $\phi:\widetilde M_\eps \to M_\eps$ and the induced metric. The
boundary of $\widetilde M_\eps$ has two connected components : one is
$\widetilde\Sigma_\eps$ which is isometric to $\Sigma_\eps$ and its mean curvature vector
points into $\widetilde M_\eps$ and the other is $\widetilde S$ which is a double cover of
$S$ and is minimal. Since $S$ is not orientable and $\widetilde S$ is a double cover,
Proposition~\ref{prop:mininonorient} gives 
\begin{equation}\label{eq:eq1}
|\widetilde S|=2|S|\ge
\boA_\boS(M)>|\Sigma_\eps|=|\widetilde \Sigma_\eps|.
\end{equation}

Since the boundary of $\widetilde M(\eps)$ is mean convex and the homology class
$[\widetilde \Sigma_\eps]$ is non zero in $H_n(\widetilde M(\eps))$,
Proposition~\ref{prop:miniorient} applies. So there is a connected orientable stable minimal
hypersurface $S'$ in $\widetilde M(\eps)$ with area less than
$|\widetilde\Sigma_\eps|=|\Sigma_\eps|$. $S'$ could be equal to $\widetilde{S}$, but this
would imply that $|\widetilde \Sigma_\eps|>|\widetilde S|$ which is not
the case by \eqref{eq:eq1}. Thus, $S'$ is in the interior of $\widetilde M_\eps$. Then
$\phi(S')$ is an embedded orientable stable minimal hypersurface in $M_\eps$ with $|\phi(S')|\le
|\Sigma_\eps|$. We then have the following inequalities $|\boA_\boS(M)|\le |\phi(S')|\le
|\Sigma_\eps|\le |\Sigma|-\delta\le |\boA_\boS(M)|-\delta$ which gives us a contradiction.
Claim~\ref{claim:1} is proved.

\begin{claim}\label{claim:2}
The hypersurface $S$ separates $M_\eps$.
\end{claim}

If $S$ does not separate, Proposition~\ref{prop:miniorient} produces a non
separating stable minimal hypersurface $S'$ in $M_\eps$ ($S'$ does not separate since it
does not vanish in homology and $M_\eps$ has only one connected component). By
Proposition~\ref{prop:openorient},
we have a manifold $\widetilde{M_\eps}$ with three boundary components $S'_1$ and
$S'_2$ isometric to $S'$ and $\widetilde\Sigma_\eps$ isometric to $\Sigma_\eps$. 

The argument is then similar to the one of Claim~\ref{claim:1}. Since the boundary of
$\widetilde M_\eps$ is mean convex, Proposition~\ref{prop:miniorient} applies to the
homology class $[\widetilde \Sigma_\eps]$ which is non zero and gives a connecteed
orientable stable minimal
hypersurface $S''$ in $\widetilde M_\eps$ whose area is less than $|\widetilde
\Sigma_\eps|=|\Sigma_\eps|$. $S''$ could be equal to $S'_i$ ($i=1,2$), but
this would imply that $|\widetilde \Sigma_\eps|>|S'_i|=|S'|\ge\boA_\boS(M)$ which is not the
case. Thus $S''$ is in the interior of $\widetilde M_\eps$. Then $\phi(S'')$ is
an embedded orientable stable minimal hypersurface in $M_\eps$ with $|\phi(S'')|\le
|\Sigma_\eps|$. We then have the following inequalities $|\boA_\boS(M)|\le |\phi(S'')|\le
|\Sigma_\eps|\le |\Sigma|-\delta\le \boA_\boS(M)-\delta$ which gives us a contradiction.
Claim~\ref{claim:2} is proved.

Thus the hypersurface $S$ is orientable and separates; let $M'$ be the piece of $M_\eps$ whose
boundary is made of $S$ and $\Sigma_\eps$. If $|S|\ge |\Sigma_\eps|$, we can apply
Proposition~\ref{prop:miniorient} to produce a stable minimal hypersurface $S'$ in the
interior of $M$ with area less than $|\Sigma_\eps|$ (we notice that $S'$ can not be equal
to $S$ since $|S'|<|\Sigma_\eps|$). We get the contradiction $\boA_\boS(M)\le
|S'|<|\Sigma_\eps|<\boA_\boS(M)$.

If $|S|<|\Sigma_\eps|$, we have $|S|<\boA_\boS(M)$. Thus $S$ is unstable, it implies that we
can apply Proposition~\ref{prop:miniorient} to produce a stable minimal hypersurface $S'$
in the interior of $M$ with area less that $|S|<|\Sigma_\eps|$ which still leads to a
contradiction as above.

So we have proved that any minimal hypersurfaces $S$ produced by the min-max theorem in
$M_\eps$ leads to a contradiction ; thus there is a continuous sweep-out as in the statement of
Proposition~\ref{prop:onesidesweep}.
\end{proof}

Let us now give the proof of Proposition~\ref{prop:sweep}.

\begin{proof}[Proof of Proposition~\ref{prop:sweep}]
Since $\Sigma$ is unstable and $|\Sigma|\le\boA_\boS(M)$, $\Sigma$ separates.
Let $M_1$ and $M_2$ be the two sides of $\Sigma$ in $M$ : $M=M_1\cup M_2$ and $M_1\cap
M_2=\Sigma$. Proposition~\ref{prop:onesidesweep} gives a continuous sweep-out
$\{\Sigma_t^1\}_{t\in[0,1]}$ of $M_1$ and a continuous sweep-out
$\{\Sigma_t^2\}_{t\in[0,1]}$ of $M_2$. We also have families $\{\Ome_t^1\}$ and
$\{\Ome_t^2\}$ of open subdomains of $M_1$ and $M_2$.

Let us define $\{\Sigma_t\}_{t\in[-1,1]}$ and $\{\Ome_t\}_{t\in[-1,1]}$ by
$\Sigma_t=\Sigma_{-t}^1$ and $\Ome_t=M_1\setminus \overline{\Ome_{-t}^1}$ if $t\le0$ and
$\Sigma_t=\Sigma_t^2$ and $\Ome_t=M_1\cup \Ome_t^2$ if $t\ge 0$. $\{\Sigma\}_{t\in[-1,1]}$
is then a sweep-out which satisfies the properties stated in Proposition~\ref{prop:sweep}.
\end{proof}

A consequence of Proposition~\ref{prop:sweep} is the following estimate of the width of a
manifold $M$.

\begin{prop}\label{prop:minowidth}
Let $M$ be a closed Riemannian $(n+1)$-manifold ($2\le n\le 6$). Let $\Sigma$ be an orientable minimal
hypersurface which is unstable and $|\Sigma|\le \boA_\boS(M)$. Then the width of $M$
satisfies $W_M\le |\Sigma|$
\end{prop}

\begin{proof}
By Proposition~\ref{prop:sweep}, there is a continuous sweep-out
$\{\Sigma_t\}_{t\in[-1,1]}$ of $M$ with $\LL(\{\Sigma_t\})=|\Sigma|$. By
Theorem~\ref{thm:conttodisc}, there is a discrete sweep-out $S\in\Pi_M$ with $\LL(S)\le
\LL(\{\Sigma_t\})=|\Sigma|$. Then $W_M\le |\Sigma|$.
\end{proof}


\section{Proof of Theorem~A}\label{sec:mainthm}

This section is entirely devoted to the proof of Theorem~A. The first step is to prove that
$\boA_1(M)$ is realized by some particular minimal hypersurfaces satisfying some
properties. The second step consists in estimating the index of these particular minimal
hypersurfaces. Let us just recall Theorem~A.

\begin{thA}
Let $M$ be an oriented closed Riemannian $(n+1)$-manifold ($2\le n\le 6$). Then
$\boA_1(M)$ is equal to one of the following possibilities.
\begin{enumerate}
\item $|\Sigma|$ where $\Sigma\in \boO$ is a min-max hypersurface of $M$ associated to the
fundamental class of $H_{n+1}(M)$ and has index $1$.

\item $|\Sigma|$ where $\Sigma\in\boO$ is stable.
\item $2|\Sigma|$ where $\Sigma\in \boU$ is stable and its orientable $2$-sheeted
cover has index $0$ or $1$; if the index is $1$, $2|\Sigma|=W_M$.
\end{enumerate}

Moreover, if $\Sigma\in\boO$ satisfies $|\Sigma|=\boA_1(M)$, then $\Sigma$ is of type $1$ or
$2$ and if $\Sigma\in\boU$ satisfies $2|\Sigma|=\boA_1(M)$, then $\Sigma$ is of type $3$.
\end{thA}

So we fix some closed orientable $(n+1)$-manifold ($2\le n\le 6$) and we look at the number
$\boA_1(M)$.


\subsection{$\boA_1(M)$ is realized}

In this section, we prove that $\boA_1(M)$ is realized either by a stable minimal
hypersurface or by an orientable min-max hypersurface. We begin by a remark about the
min-max hypersurfaces.

The Almgren-Pitts theory tells that the width $W_M$ of the manifold is equal to
$\sum_{i=1}^p n_i|S_i|$ where $S_1,\cdots,S_p$ is a finite collection of connected minimal
hypersurfaces and $n_1,\cdots,n_p$ are integers (Theorem~\ref{thm:discminmax}). The
following proposition makes this writing more precise when $W_M\le \boA_\boS(M)$.

\begin{prop}\label{prop:uniqminmax}
Let us consider a writing $W_M=\sum_{i=1}^p n_i|S_i|$ given by Theorem~\ref{thm:discminmax}.
If $W_M\le \boA_\boS(M)$ then
\begin{itemize}
\item either $W_M=|S_1|$ with $S_1\in\boO$,
\item or $W_M=2|S_1|$ with $S_1\in \boU$.
\end{itemize}
Moreover, if $W_M<\boA_\boS(M)$, the second case is not possible.
\end{prop}

\begin{proof}
We know $W_M=\sum_{i=1}^p n_i|S_i|$. Let us first assume that $S_1$ is an
orientable minimal hypersurface. If $S_1$ is stable then $\boA_\boS(M)\le |S_1|\le
\sum_{i=1}^p n_i|S_i|
=W_M \le \boA_\boS(M)$. So we have equality in all the inequalities and $n_1=1$ and $p=1$. If
$S_1$ is unstable, we have $|S_1|\le W_M\le \boA_\boS(M)$ and, by
Proposition~\ref{prop:minowidth}, $W_M\le |S_1|\le \sum_{i=1}^p n_i|S_i| =W_M$ so
$n_1=1$ and $p=1$.

Let us now assume that $S_1$ is non orientable, we then know by Proposition~6.1 in
\cite{Zho} that $n_1$ is at least $2$. This implies that $\boA_\boS(M)\le 2|S_1|\le W_M\le
\boA_\boS(M)$ and then $n_1=2$ and $p=1$.
\end{proof}

The proof of Theorem~A consists in proving that
\begin{equation}\label{eq:boa1}
\boA_1(M)=\min (\boA_S(M),W_M).
\end{equation}

Because of Propositions~\ref{prop:ministable} and \ref{prop:uniqminmax}, the above
inequality implies that $\boA_1(M)$ is realized. So let us prove \eqref{eq:boa1}. By
Proposition~\ref{prop:ministable}, $\boA_S(M)$ is realized (if it is finite); so assume
that $\boA_S(M)>\boA_1(M)$. By Propositions~\ref{prop:miniorient} and
\ref{prop:mininonorient}, it means that there is some orientable unstable minimal
hypersurface $\Sigma$ with $|\Sigma|<\boA_\boS(M)$. By Proposition~\ref{prop:minowidth},
$\boA_\boS(M)>|\Sigma|\ge W_M$ so Proposition~\ref{prop:uniqminmax} applies and $W_M$ is
realized by a connected minimal hypersurface $\barre S$. We have then proved that any minimal
hypersurface $\Sigma$ with $|\Sigma|<\boA_\boS(M)$ is such that $|\barre S|=W_M\le
|\Sigma|$; so \eqref{eq:boa1} is proved.

Now let us consider a minimal hypersurface $\Sigma$ that realizes $\boA_1(M)$ but not of
type 2 or 3, \textit{i.e.} not stable. We want to prove that $\Sigma$ is an orientable min-max hypersurface. By Proposition~\ref{prop:mininonorient}, $\Sigma$ is
orientable. By Proposition~\ref{prop:sweep}, there is
a continuous sweep-out $\{\Sigma_t\}_{t\in[-1,1]}$ of $M$ with $\Sigma_0=\Sigma$ and
$\LL(\{\Sigma_t\})=|\Sigma|$. By Theorem~\ref{thm:conttodisc}, there is a discrete
sweep-out $S=\{\phi_i\}$ associated to $\{\Sigma_t\}$ with $\LL(S)\le \LL(\{\Sigma_t\})$.
As a consequence, $W_M\le \LL(S)\le \LL(\{\Sigma_t\})=|\Sigma|=\boA_1(M)\le W_M$ ; thus,
$S$ realizes
the width of $M$. So there is a min-max sequence $\{\phi_{i_j}(x_j)\}$ that converges in
the varifold sense to a minimal hypersurface that realizes the width of $M$. We want to
prove that $\Sigma$ is this limit minimal hypersurface.

In order to use Remark~\ref{rmq:masscontrol}, let us denote
$\widetilde\Phi(t)=\Sigma_{-1+2t}$. We know that $\lim_j
\MM(\phi_{i_j}(x_j))=W_M=|\Sigma|=\MM(\widetilde\Phi(1/2))$. So, because of
\eqref{eq:masscontrol} and the properties of the continuous sweep-out $\{\Sigma_t\}_t$,
$x_j\to 1/2$. By Remark~\ref{rmq:masscontrol}, this implies that $\phi_{i_j}(x_j)$
converges to $\widetilde\Phi(1/2)=\Sigma$ in the flat topology. Since $|\Sigma|=\lim_j
\MM(\phi_{i_j}(x_j))$, this implies that we also have convergence in varifold sense. So
$\Sigma$ is the limit of a min-max
sequence and then a min-max hypersurface.

In order to finish the proof of Theorem, we still have to control the index of these
hypersurfaces.


\subsection{Index in the orientable case}\label{sec:indexorient}

Let us now prove that a type 1 hypersurface has index $1$ (see also \cite{MaNe3}).

Let $\Sigma$ be an orientable unstable minimal hypersurface with
$|\Sigma|=W_M=\boA_1(M)$. We want to prove that its index is at most $1$. So let us
assume it has index at least $2$. We then denote by $u_1$ and $u_2$ the first two
eigenfunctions of the Jacobi operator on $\Sigma$. By Proposition~\ref{prop:sweep}, there is a
sweep-out $\{\Sigma_t\}_{t\in[-1,1]}$ of $M$ such that $\Sigma_0=\Sigma$,
$\LL(\{\Sigma_t\})=|\Sigma|$ and $|\Sigma_t|\le |\Sigma|-\delta(\eps)$ for any $|t|\ge
\eps$. Moreover we have $\Sigma_t=\Phi(\Sigma,t)$ for $t$ close to $0$ where
$$
\Phi : \Sigma\times \R \to M; (p,t)\mapsto \exp_p(tu_1(p)\nu(p)).
$$

Let us change the definition of the map $\Phi$ by adding one variable and consider the new
definition
$$
\Phi : \Sigma\times \R\times \R \to M; (p,t,s)\mapsto \exp_p((tu_1(p)+su_2(p))\nu(p)).
$$
For $t$ and $s$ small, we define $\Sigma_{t,s}=\Phi(\Sigma,t,s)$. These are embedded
hypersurfaces living in a tubular neighborhood of $\Sigma$. The volume functional
$A(t,s)=|\Sigma_{t,s}|$ is smooth for $t,s$ small and its differential at $(0,0)$ vanishes
since $\Sigma$ is minimal. Its Hessian is negative definite since $u_1$ and $u_2$ are
associated to negative eigenvalues of the Jacobi operator on $\Sigma$. So for $\eps$ small
enough, we have $A(\eps\sin \theta,\eps\cos \theta)\le |\Sigma|-c\eps^2$ for some $c>0$
and all $\theta\in\R$.

Let us define a new continuous sweep-out $\{\Sigma'_t\}_{t\in[-1,1]}$ of $M$ by the
following choices
$$
\Sigma'_t=
\begin{cases}
\Sigma_t \text{ if }t\le -\eps\\
\Sigma_{\eps\sin\frac{t\pi}{2\eps},\eps\cos\frac{t\pi}{2\eps}} \text{ if }-\eps\le t\le \eps\\
\Sigma_t \text{ if }t\ge \eps
\end{cases}
$$
The family of open subsets $\{\Ome_t'\}_t$ associated to $\{\Sigma_t'\}_t$ can be adapted
from the original family $\{\Ome_t\}_t$.

Because of the properties of the original sweep-out and the control on the function $A$,
we see that $|\Sigma'_t|\le |\Sigma|-\delta$ for some $\delta>0$ and any $t\in[-1,1]$.
By Theorem~\ref{thm:conttodisc}, there exists a discrete sweep-out $S\in\Pi_M$ with
$\LL(S)\le |\Sigma|-\delta$. This implies that $W_M\le |\Sigma|-\delta=W_M-\delta$ and
gives a contradiction. So the index of $\Sigma$ is at most $1$.


\subsection{Index in the non orientable case}
In this section, we control the index of the double cover of a type 3 non orientable
minimal hypersurface that realizes $\boA_1(M)$. We want to prove that it has index at most
$1$. 

Let $\Sigma$ be a type 3 non orientable minimal hypersurface. We thus have
$2|\Sigma|=\boA_1(M)\le W_M$. We open $M$ along $\Sigma$
by Proposition~\ref{prop:opennonorient} and get $\phi:\widetilde M\to M$ where $\phi
:\widetilde \Sigma=\partial \widetilde M\to \Sigma$ is a double cover. We lift the metric
of $M$ to $\widetilde M$. Let $\sigma$ denote the non trivial deck transformation of
$\phi:\widetilde \Sigma\to \Sigma$.

We assume that the Jacobi operator on $\widetilde\Sigma$ has index at least $2$. 
We know that $\Sigma$ is a stable minimal hypersurface means that the Jacobi operator
on $\widetilde\Sigma$ is positive on the space of $\sigma$-odd functions. So 
the first two eigenfunctions $u_1$ and $u_2$ on $\widetilde\Sigma$ must be $\sigma$-even
since their eigenvalues are negative. As a consequence, $u_1$ and $u_2$ can be seen as
functions on $\Sigma$.

Since $\phi$ is a local isometry from the interior of $\widetilde M$ to $M\setminus
\Sigma$, $\boA_\boS(\widetilde M)\ge \boA_\boS(M)$ and thus
$|\widetilde\Sigma|=\boA_1(M)\le\boA_\boS(M)\le
\boA_\boS(\widetilde M)$. Thus Proposition~\ref{prop:onesidesweep} gives 
a continuous sweep-out $\{\widetilde\Sigma_t\}_{t\in[0,1]}$ of $\widetilde M$ with
$\wtilde\Sigma$ of maximum area.

If $\Phi: \widetilde\Sigma\times [0,\eps]\to \widetilde M; (p,t)\mapsto
\exp_p(tu_1(p)\tilde \nu(p))$ ($\tilde \nu$ the inward unit normal vector field to
$\widetilde\Sigma$), we know that
$\widetilde\Sigma_t=\Phi(\widetilde\Sigma,t)$ for $t$ close to $0$. Moreover, for $t>0$,
we have $\widetilde \Sigma_t=(\partial\widetilde\Ome_t\setminus \widetilde \Sigma) \cup
\widetilde P_t$ where $\{\wtilde\Ome_t\}$ is a family of open subsets of $\widetilde M$ with
$\widetilde \Sigma\subset \widetilde\Ome_t$ and $\{\widetilde P_t\}$ is a family of finite
subsets. 

Let us consider $\Ome_t=\phi(\widetilde\Ome_t)$ and $P_t=\phi(\wtilde P_t)$ for
$t\in[0,1]$. We have $\Ome_0=\emptyset$
and, for $t>0$, $\Ome_t$ is a domain in $M$ that contains $\Sigma$ and whose boundary is
$\Sigma_t\setminus P_t=\phi(\widetilde\Sigma_t\setminus \wtilde P_t)$. For $t$ close to
$0$, $\Ome_t$ is contained in a
tubular neighborhood of $\Sigma$.

Let $N\Sigma$ be the normal bundle to $\Sigma$, it is a twisted line bundle over $\Sigma$.
We notice that the map $\phi:\widetilde\Sigma$ extends to a double cover
$\pi: N\widetilde\Sigma\to N\Sigma$ where the normal bundle $N\widetilde\Sigma$ to
$\widetilde\Sigma$ is trivial. For a non negative function $u$ on $\Sigma$, we can consider 
$$
N_u\Sigma=\{(p,n)\in N\Sigma \mid \|n\|<u(p)\}
$$
If $\eps>0$, we have $N_\eps\Sigma$ for the constant function $u\equiv \eps$. The
map $\Psi : N_\eps\Sigma\to M; (p,n(p))\mapsto \exp_p(n(p))$ is a diffeomorphism on the
$\eps$-tubular neighborhood of $\Sigma$ when $\eps$ is small enough. For a continuous non
negative function $u$ on $\Sigma$ with $u\le \eps$, $D_u=\Psi(N_u\Sigma)$ is 	an open
subset of the $\eps$-tubular neighborhood of $\Sigma$. With this notation, if $0<t<\eps'$
for $\eps'$ small enough, we have $\Ome_t=D_{tu_1}$ (here the $\sigma$-even function $u_1$
is seen as a function on $\Sigma$).

In order to construct a particular sweep-out, we are going to change the domains $\Ome_t$
for $t$ small. Let $\eps'$ be such that $\eps'(\|u_1\|_\infty+\|u_2\|_\infty)\le \eps$.
Then if $t\le \eps'$ and $\theta\in\R$, the domain $O_{t,\theta}=D_{t(\cos\theta
u_1+\sin\theta u_2)^+}$ (where $u^+=\max(0,u)$ denotes the positive part of $u$) is well
defined and is included in the $\eps$-tubular neighborhood $D_\eps$ of $\Sigma$. 

Let us remark that
$u_2$ does not have a fixed sign so $\cos\theta u_1+\sin\theta u_2$ can be negative
somewhere and then $\Sigma$ can be not included in $O_{t,\theta}$. The boundary of
$O_{t,\theta}$ is included in an immersed hypersurface $S_{t,\theta}$ which is the image
of $\{p,t(\cos\theta u_1(p)+\sin\theta u_2(p))\tilde \nu(p),p\in\widetilde\Sigma\}\in
N\widetilde\Sigma$ by $\Psi\circ\pi$. This implies that $O_{t,\theta}$ is a domain with
rectifiable boundary. Moreover we can estimate $\boH^n(\partial O_{t,\theta})$ in two
different ways.

The first estimation is just the fact that $\partial O_{t,\theta}\subset 
S_{t,\theta}$ so 
\begin{equation}\label{eq:graph}
\boH^n(\partial O_{t,\theta})\le |S_{t,\theta}|
\end{equation}
($|S_{t,\theta}|$ computes the volume of an immersed hypersurface so multiplicities may
appear).

The second estimation uses the fact that $\cos\theta u_1+\sin\theta u_2$ can be negative
somewhere. So, in order to compute the $\boH^n$-measure of $\partial O_{t,\theta}$, we just
have to take care of the part of $S_{t,\theta}$ that correspond to point where
$\cos\theta u_1+\sin \theta u_2$ is positive. This implies that 
\begin{equation}\label{eq:neg}
\boH^n(\partial O_{t,\theta})\le 2\boH^n(\{p\in\Sigma\mid \cos\theta u_1(p)+\sin \theta
u_2(p)>0\})+ct
\end{equation}
for some constant $c$ that does not depend on $t$ and $\theta$.

As in Section~\ref{sec:indexorient}, the fact that $u_1$ and $u_2$ are eigenfunctions
associated to negative eigenvalues of the Jacobi operator on $\widetilde\Sigma$ implies
that there is some positive constant $c'$ such that, for $t$ small,
\begin{equation}\label{eq:unstable}
|S_{t,\theta}|\le|\widetilde \Sigma|-c't^2=2|\Sigma|-c't^2.
\end{equation}

Let us define our particular "sweep-out". So choose some small $\eta>0$ such that, for
$0<t<\eta$, the subdomains $O_{t,\theta}$ are well defined and the estimates
\eqref{eq:graph}, \eqref{eq:neg} and \eqref{eq:unstable} are true. For $t\in[\eta,1]$, we
define $\Ome'_t=\Ome_t$ and, for $t\in[-\pi/2+\eta,\eta]$, we define $\Ome'_t=O_{\eta,\eta-t}$
(both definitions coincide at $t=\eta$, see Figure~\ref{fig:fig}). We then have
$\Ome'_{-\pi/2+\eta}=O_{\eta,\pi/2}$. Finally, for $t\in[-\pi/2,-\pi/2+\eta]$, we define
$\Ome'_t=O_{t+\pi/2,\pi/2}$, we notice that both definitions agree at $t=-\pi/2+\eta$ and
$\Ome_{-\pi/2}'=\emptyset$. We notice that the family of open subsets
$\{\Ome_t'\}_{t_\in[-\pi/2,1]}$ satisfies (sw2) and (sw3).

\begin{figure}[h]
\begin{center}
\resizebox{1\linewidth}{!}{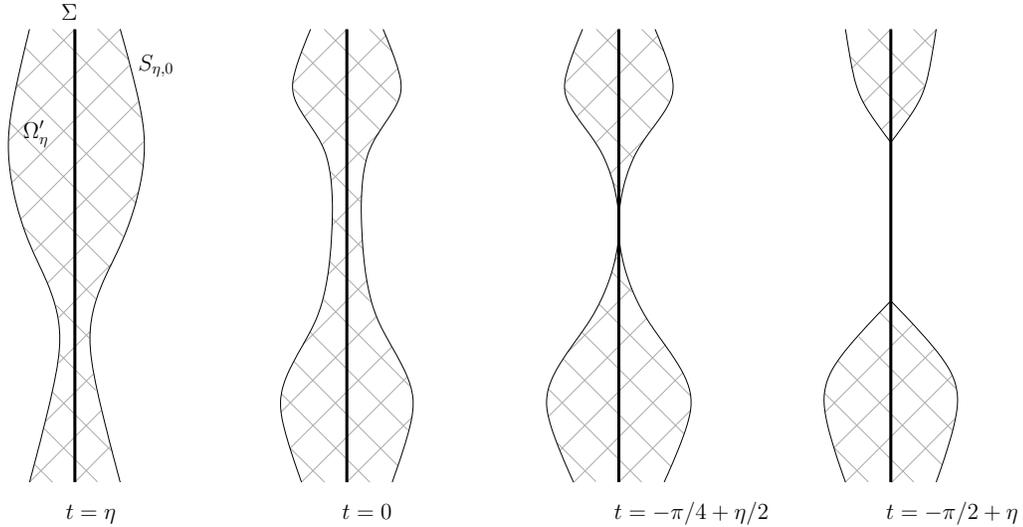}
\caption{The evolution of $\Ome'_t$ for $t\in[-\pi/2+\eta,\eta]$\label{fig:fig}}
\end{center}
\end{figure}

Let us estimate the mass of $\partial[\Ome'_t]$ for $t\in[-\pi/2,1]$. If $t\ge \eta$,
$\Ome_t'=\Ome_t$ so we know by Proposition~\ref{prop:onesidesweep} that there is $\delta$
such that
\begin{equation}\label{eq:est1}
\MM(\partial[\Ome_t'])=|\partial\Ome_t|=|\Sigma_t|\le|\widetilde\Sigma|-\delta=2|\Sigma|-\delta.
\end{equation}
For $t\in[-\pi/2+\eta,\eta]$, we use \eqref{eq:graph} and \eqref{eq:unstable} to obtain
\begin{equation}\label{eq:est2}
\MM(\partial[\Ome'_t])=\boH_n(\partial\Ome'_t)\le |S_{\eta,\eta-t}|\le
2|\Sigma|-c'\eta^2.
\end{equation}
For $t\in[-\pi/2,\pi/2+\eta]$, we use \eqref{eq:graph}, \eqref{eq:neg} and
\eqref{eq:unstable} to obtain
\begin{equation}\label{eq:est3}
\begin{split}
\MM(\partial[\Ome'_t])&=\boH_n(\partial\Ome'_t)\\
&\le \min \{2|\Sigma|-c'(t+\pi/2)^2,2\boH^n(\{ u_2>0\})+c(t+\pi/2)\})
\end{split}
\end{equation}

Since $\boH^n(\{u_2>0\})<|\Sigma|$, the estimates \eqref{eq:est1}, \eqref{eq:est2} and
\eqref{eq:est3} imply that there is some $\delta'>0$ such that
$\MM(\partial[\Ome'_t])\le 2|\Sigma|-\delta'$ for any $t\in[-\pi/2,1]$. We can now apply
Theorem~\ref{thm:conttodisc} to obtain a discrete sweep-out $S\in\Pi_M$ with $\LL(S)\le
2|\Sigma|-\delta'$ (the hypothesis on $\bfm(\Phi,r)$ is fulfilled since the supports of
$\partial[\Ome_t']$ are contained in a continuous family of immersed hypersurfaces so
arguments in the proof of Proposition~5.1 in \cite{Zho} apply, see also
Remark~\ref{rmq:noconcentr}). As a consequence, this implies that $W_M\le 2|\Sigma|-\delta'$
which contradicts that $2|\Sigma|\le W_M$. We have then proved that the Jacobi operator on
$\widetilde\Sigma$ has index at most $1$. 

If $\Sigma$ has index $1$, as above, $\Sigma_t$ and $\Ome_t$ can be constructed since they
only depend on the first eigenvalue. By construction
$\sup\{\MM(\partial[\Ome_t]),t\in[0,1]\}=2|\Sigma|$. So by Theorem~\ref{thm:conttodisc},
$W_M\le2|\Sigma|$ and then $W_M=2|\Sigma|$.

\begin{remarq}
We just proved that $W_M=2|\Sigma|$ not that $\Sigma$ is a min-max hypersurface
\textit{i.e.} a varifold limit of a min-max sequence. With respect to the orientable
case, the difference comes from the fact that, as currents, the limit is here $0$. In
fact, it seems possible, looking at the proof of Theorem~\ref{thm:conttodisc} to control
the support of the discrete sweep-out from the support of the continuous one and prove
that actually $\Sigma$ is a min-max hypersurface associated to the fundamental class of
$H_{n+1}(M)$.
\end{remarq}


\section{The $3$-dimensional case}\label{sec:mainthm3}

In this section, we give some improvements to Theorem~A when the ambient
manifold has dimension $3$.


\subsection{Some topology of $3$-manifolds}

Let us recall some definitions about the topology of $3$-manifolds. 

A \textit{compression body} is a $3$-manifold $B$ with boundary with a particular boundary
component $\partial_+B=\Sigma\times\{0\}$ such that $B$ is obtained from $\Sigma\times
[0,1]$ by attaching $2$-handles and $3$-handles, where no attachments are performed along
$\partial_+B=\Sigma\times\{0\}$ (see \cite{CaGo} for related definitions).

A compression body with only one boundary component, \textit{i.e.} $\partial
B=\partial_+B$, is called a \textit{handlebody}. A handlebody can also be seen as a closed
ball with $1$-handles attached along the boundary. 


Let $M$ be a compact $3$-manifold with maybe non empty boundary. A separating orientable
surface $\Sigma$ in the interior of $M$ is a \textit{Heegaard splitting} if both sides of
$\Sigma$ are compression bodies $B_1$ and $B_2$ with $\partial_+ B_1=\Sigma=\partial_+
B_2$. Let us notice that Heegaard splittings always exist. If $M$ has no boundary $B_1$
and $B_2$ are handlebodies.

If $M$ is a compact $3$-manifold, the \textit{Heegaard genus} of $M$ (denoted by $g_H(M)$) is
defined as the minimum of the genera of all surfaces that are Heegaard splittings.

In the following, we will use the following characterization of handlebodies which is due
to Meeks, Simon and Yau (see Proposition~$1$ in \cite{MeSiYa}).

\begin{prop}\label{prop:mesiya}
Let $M$ be a compact Riemannian $3$-manifold with one boundary component. $M$ is a
handlebody if and only if the isotopy class of a parallel surface to $\partial M$ contains
surfaces of arbitrary small area.
\end{prop}

Let $M$ be a compact $3$-manifold with boundary, a proper embedding of $\S^1\times[0,1]$
is an incompressible annulus if the inclusion is $\pi_1$ injective (by proper we mean
$\partial(\S^1\times[0,1])$ is sent to $\partial M$).


\subsection{An improvement in the $3$-dimensional case}

The improvement that we obtain in the $3$-dimensional case is that we can control the
genus of the min-max surfaces that appear in Theorem~A. So we have the
following result.

\begin{thB}
Let $M$ be an oriented closed Riemannian $3$-manifold. Then $\boA_1(M)$ is equal to
one of the following possibilities.
\begin{enumerate}
\item $|\Sigma|$ where $\Sigma\in \boO$ is a min-max surface of $M$ associated to the
fundamental class of $H_{3}(M)$, $\Sigma$ has index $1$ and $g_\Sigma\ge g_H(M)$.
\item $|\Sigma|$ where $\Sigma\in\boO$ is stable.
\item $2|\Sigma|$ where $\Sigma\in \boU$ is stable and its orientable $2$-sheeted
cover has index $0$ or $1$. Moreover if the double cover $\widetilde\Sigma$ has index $1$,
we have $g_{\widetilde\Sigma}\ge g_H(M)-1$ and $W_M=2|\Sigma|$.
\end{enumerate}

Moreover, if $\Sigma\in\boO$ satisfies $|\Sigma|=\boA_1(M)$, then $\Sigma$ is of type $1$ or
$2$ and if $\Sigma\in\boU$ satisfies $2|\Sigma|=\boA_1(M)$, then $\Sigma$ is of type $3$.
\end{thB}

The proof is based on the following lemma where we use ideas similar to the proof of
Proposition~\ref{prop:onesidesweep}.

\begin{lem}\label{lem:handle}
Let $M$ be a compact Riemannian $3$-manifold with $\partial M=\Sigma$ connected, minimal
and unstable. Moreover we assume that $|\Sigma|\le \boA_\boS(M)$. Then $M$ is a
handlebody. 
\end{lem}

\begin{proof}
Since $\Sigma$ is unstable, using the notations of the proof of
Proposition~\ref{prop:onesidesweep}, the manifold $M_t=M\setminus
\Phi(\Sigma\times[0,t)])$ has mean convex boundary for $t$ small. Let $t_0>t$ be small and
look at the quantity
$$
A=\inf\{|S|,\, S\text{ isotopic to }\Sigma_{t_0}\text{ in }M_t\}.
$$
If $A=0$, then $M_t$ and thus $M$ are handlebodies by Proposition~\ref{prop:mesiya}. If
$A\neq 0$, $A$ is realized by a union of stable minimal surfaces with multiplicities (see
Theorem~1' in \cite{MeSiYa}). Let $S$ be one of these stable minimal surfaces. If $S\in
\boO$, then $\boA_\boS(M)\le |S|\le |\Sigma_{t_0}|<|\Sigma|\le \boA_\boS(M)$ which is a
contradiction. If $S\in \boU$, Theorem~1' in \cite{MeSiYa} tells that its multiplicity
is at least $2$, so the same contradiction as above occurs. So we have $A=0$ and $M$ is a
handlebody. 
\end{proof}

\begin{proof}[Proof of Theorem~B]
The only thing we need is the control on the genus of the surface $\Sigma$.

Let $\Sigma$ be a type $1$ surface. So $\Sigma$ is a non stable minimal surface and
$|\Sigma|=\boA_1(M)\le \boA_\boS(M)$. By Proposition~\ref{prop:miniorient}, $\Sigma$
separates $M$ and, by Lemma~\ref{lem:handle}, both sides of $\Sigma$ in $M$ are
handlebodies. $\Sigma$ is then a Heegaard splitting and then $g_\Sigma\ge g_H(M)$.

Let $\Sigma$ be a type $3$ surface whose double cover $\widetilde\Sigma$ is not stable.
Let us open $M$ along $\Sigma$ (Proposition~\ref{prop:opennonorient}) to obtain a
$3$-manifold $\widetilde M$ with boundary $\widetilde\Sigma$. Since $\Sigma$ realizes
$\boA_1(M)$ we have $|\widetilde\Sigma|\le \boA_\boS(\widetilde M)$. By
Lemma~\ref{lem:handle}, $\widetilde M$ is a handlebody. So $M$ can be seen as a handlebody
where points on the boundary are identified through a fixed point free involution that
reverses the orientation. Actually it is possible to control the Heegaard genus of $M$
in terms of the genus of $\widetilde\Sigma$: there is an argument attributed to Rubinstein
by Shalen (see 4.5 in\cite{Sha}) which implies that $g_H(M)\le g_{\widetilde\Sigma}+1$.
The argument works as follows. Let $M_\eps$ be the outside of a $\eps$-tubular
neighborhood of $\Sigma$. Since $\wtilde M$ is a handlebody, $M_\eps$ is also a
handlebody. Choose a point $p$ on $\Sigma$ and consider $\gamma$ the normal geodesic to
$\Sigma$ with length $2\eps$ and $p$ as middle point. The end points of $\gamma$ are in
$\partial M_\eps$. Let $H$ be the union of $M_\eps$ with a small tubular neighborhood of
$\gamma$. $H$ can be seen as $M_\eps$ to which a $1$-handle is attached so it is a
handlebody. In fact the complement of $H$ is also a handlebody since the complement of a
point in a closed surface continuously retract to a bouquet of circles. Now the genus of
$\partial H$ is just $g_{\wtilde\Sigma}+1$.
\end{proof}


\section{Minimal surfaces in hyperbolic $3$-manifolds}\label{sec:hyperbolic}

In this section, we prove a lower bound for the area of minimal surfaces in hyperbolic
$3$-manifolds. 


\subsection{Area and genus}

In a hyperbolic $3$-manifold, the area of a minimal surface $\Sigma$ is always bounded
above by its topology, we have $|\Sigma|\le -2\pi \chi(\Sigma)$ (it is a classical
consequence of the The Gauss and Gauss-Bonnet formulas). If its index is at most
$1$, we can also obtain a lower bound in terms of its genus.
\begin{lem}\label{lem:area}
Let $\Sigma$ be an immersed orientable closed minimal surface in an oriented hyperbolic
$3$-manifold.
\begin{itemize}
\item If $\Sigma$ is stable, then $|\Sigma|\ge \pi|\chi(\Sigma)|=2\pi(g_\Sigma-1)$.
\item If $\Sigma$ has index $1$, then
$|\Sigma|\ge2\pi\Big(g_\Sigma-2-\left[\frac{g_\Sigma+1}2\right]\Big)$
\end{itemize} 
\end{lem}

The first estimate tells that the area of an orientable stable minimal surface is well
controlled by its topology $\pi|\chi(\Sigma)|\le |\Sigma|\le 2\pi|\chi(\Sigma)|$. This
estimate was observed by K. Uhlenbeck but not published (see Hass \cite {Has}).

\begin{proof}
If $\Sigma$ is stable, we can use the constant function $1$ as a test function in the
stability operator and obtain
$$
\int_\Sigma-(\Ric(\nu,\nu)+\|A\|^2)\ge 0
$$
The Gauss formula implies that $\|A\|^2=-2(K_\Sigma+1)$ with $K_\Sigma$ the sectional
curvature of $\Sigma$. So, using the Gauss-Bonnet formula, we obtain
$$
|\Sigma|\ge-\frac12\int_\Sigma K_\Sigma=-\pi\chi(\Sigma)=2\pi(g_\Sigma-1)
$$

The study of the index $1$ case is based on what is called the Hersch trick. Let $u_1$ be
the first eigenfunction of the Jacobi operator on $\Sigma$. Let $\phi$ be a conformal map
from $\Sigma$ to $\S^2\subset \R^3$ and look at the following integral
$$
\int_\Sigma u_1(p) \times h\circ \phi(p) dp\in \R^3
$$
where $h$ is a M\"obius tranformation of $\S^2$. Since $u_1$ is non negative, we can find $h$
such that the above integral vanishes (see \cite{LiYa}). Let $(f_1,f_2,f_3)$ be the three
coordinates of $h\circ \phi$. $f_i$ is then orthogonal to $u_1$ and $\Sigma$ has index $1$, so
$$
\int_\Sigma \|\nabla f_i\|^2-(\Ric(\nu,\nu)+\|A\|^2)f_i^2\ge 0.
$$
Summing these three inequalities and using that $h\circ \phi$ is conformal we get
\begin{align*}
0\le \int_\Sigma& \|\nabla h\circ\phi\|^2-(\Ric(\nu,\nu)+\|A\|^2)\\
&=8\pi\deg(h\circ\phi)-\int_\Sigma (\Ric(\nu,\nu)+\|A\|^2)\\
&=8\pi\deg(\phi)-\int_\Sigma (\Ric(\nu,\nu)+\|A\|^2).
\end{align*}
As in \cite{RiRo}, we can choose $\phi$ such that $\deg(\phi)\le
1+\left[\frac{g_\Sigma+1}2\right]$. So computations similar to the stable case give
$$
|\Sigma|\ge 2\pi\Big(-1-\left[\frac{g_\Sigma+1}2\right]\Big)-\frac12\int_\Sigma
K_\Sigma=2\pi\Big(g_\Sigma-2-\left[\frac{g_\Sigma+1}2\right]\Big).
$$
\end{proof}

We remark that in the above proof we only use the fact that the sectional curvature of the
ambient manifold is bounded below by $-1$.

We can also remark that, in the stable case, the equality can not occur. Indeed, if
$|\Sigma|=2\pi(g_\Sigma-1)$, the proof tells that the constant function $1$ is in the
kernel of the Jacobi operator so $\Ric(\nu,\nu)+\|A\|^2=0$ and then $K_\Sigma=-2$. So the
lift of $\Sigma$ to $\H^3$ gives a complete immersion with constant sectional curvature
$-2$ which is not possible by Theorem~12 in \cite{Gal}.


\subsection{The compact case}

We can now state our lower bound for the area of minimal surfaces in hyperbolic
$3$-manifolds.

\begin{thC}
Let $M$ be a closed orientable hyperbolic $3$-manifold. If $g_H(M)\ge 7 $ then $\boA_1(M)\ge
2\pi$. In other words, any orientable minimal surface in $M$ has area at least $2\pi$ and
any non orientable minimal surface has area at least $\pi$.
\end{thC}

\begin{proof}
Since $M$ has negative sectional curvature, any immersed closed minimal surface in $M$ has
negative Euler characteristic. By Theorem~B, $\boA_1(M)$ is realized by some
minimal surface $\Sigma$.

If $\Sigma$ is of type $2$, Lemma~\ref{lem:area} gives $|\Sigma|\ge 2\pi(g_\Sigma-1)\ge
2\pi$ since $\Sigma$ has negative Euler characteristic.

If $\Sigma$ is of type $1$, Lemma~\ref{lem:area} gives
$$
|\Sigma|\ge 2\pi\Big(g_\Sigma-2-\left[\frac{g_\Sigma+1}2\right]\Big)\ge
2\pi\Big(g_H(M)-2-\left[\frac{g_H(M)+1}2\right]\Big)\ge 2\pi.
$$

If $\Sigma$ is of type $3$, let $\widetilde\Sigma$ be its orientable double cover. If
$\widetilde\Sigma$ is stable, we get $2|\Sigma|=|\widetilde\Sigma|\ge 2\pi$ as above. If
$\widetilde\Sigma$ has index $1$, Theorem~B gives us $g_{\wtilde\Sigma}\ge
g_H(M)-1$ and we have
\begin{align*}
2|\Sigma|=|\widetilde\Sigma|&\ge
2\pi\Big(g_{\widetilde\Sigma}-2-\left[\frac{g_{\widetilde\Sigma}+1}2\right]\Big)\\
&\ge 2\pi\Big(g_H(M)-3-\left[\frac{g_H(M)}2\right]\Big)\ge 2\pi.
\end{align*}
So in all cases, we have $\boA_1(M)\ge 2\pi$.
\end{proof}

\begin{remarq}
If we know that there is no non orientable surface in $M$, then the conclusion of the
above theorem is true if we only assume $g_H(M)\ge 6$.

We can also remark that the same result is true if we only assume that the sectional
curvature of $M$ satisfies $-1\le K_M<0$.

We also notice that the existence of hyperbolic $3$-manifolds with arbitrarily large
Heegaard genus is given by a result of Souto (see Theorem~4.1 in \cite{Sout} and \cite{NaSo}).
\end{remarq}

If the hypothesis on the Heegaard genus is dropped, the monotonicity formula and the
thin-thick decomposition of $M$ tells us that any minimal surface in a closed hyperbolic
$3$-manifold has area at least some $c>0$ that does not depend on $M$ (see
\cite{CoHaMaRo}) (this is also true for closed immersed $H$-surfaces with $H<1$). So this leads us to ask: what is a closed orientable hyperbolic
$3$-manifold $M$ that minimizes $\boA_1(M)$ among such $3-$manifolds? What is a minimal
surface that realizes this $\boA_1(M)$ in $M$? We ask the same question for properly
embedded minimal surfaces in complete hyperbolic $3$-manifolds of finite volume $M$ (see
the following section). We believe an answer is a Seifert surface (a once punctured torus)
of the figure eight knot, made minimal in the hyperbolic structure of the complement of
the figure eight knot.


\subsection{The finite volume case}

In this section we extend Theorem~C to the case where $M$ is a complete non-compact
hyperbolic $3$-manifold with finite volume. Notice that such a manifold has closed minimal
surfaces (see \cite{CoHaMaRo}).

If $M$ is such a manifold, $M$ is diffeomorphic to the interior of a compact manifold
$\barre M$ with
boundary whose boundary components are tori. Moreover, each end $E$ of $M$ can be isometrically
parametrized by $N_{v_1,v_2}$, the quotient of $\{(x,y,z)\in\R^2\times\R_+^*, z\ge 1/2\}$
by the group generated by the translations by the independent horizontal vectors $v_1$
and $v_2$, endowed with the Riemannian metric
$$
g_\H=\frac1{z^2}(dx^2+dy^2+dz^2).
$$
We notice that the $z$ coordinate is well defined on $N_{v_1,v_2}$.

In the following, we denote $\Lambda(E)=\Lambda(N_{v_1,v_2})=\max (\|v_1\|,\|v_2\|)$
($\|\cdot\|$ the Euclidean
norm) and we notice that by parametrizing a smaller part of $E$ we can always choose a
chart with $\Lambda(E)$ as small as we want.

We will use other metrics on $N_{v_1,v_2}$ to change the metric on $M$. More precisely, we
will use the following metric
$$
g_\Psi=\frac1{\Psi^2(z)}(dx^2+dy^2+dz^2)
$$
where $\Psi$ is function satisfying
\begin{itemize}
\item $\Psi(z)=z$ on $[1/2,1]$,
\item $\Psi$ is non decreasing.
\end{itemize}
The first condition means that this metric can be glued to the original hyperbolic metric. The
second one gives that the foliation by the tori $T(c)=\{z=c\}$ has a mean curvature vector
pointing in the $\partial_z$ direction.

In \cite{CoHaMaRo}, Collin, Hauswirth and the authors proved the following result.

\begin{prop}\label{prop:max}
Let $t_0\in[1/2, 1]$ and $\Psi$ be as above. There is a $\Lambda_0=\Lambda(t_0,\Psi)$ such
that if $\Lambda(N_{v_1,v_2})\le \Lambda_0$ and $\Sigma$ is a compact embedded minimal
surface in $(N_{v_1,v_2},g_\Psi)$ with $\partial\Sigma\in T(1-t_0)$ then $\Sigma\subset\{z\le
1\}$.
\end{prop}

As said above, in a finite volume hyperbolic $3$-manifold $M$, we can choose a chart
$N_{v_1,v_2}$ of each end $E$ with $\Lambda(E)\le \Lambda(1/3)$ (here $\Psi(z)=z$). The
above proposition says that any compact minimal surface in $M$ never enters in $\{z>1\}$
inside the ends. Thus all compact minimal surfaces in $M$ stay in a compact piece
of $M$; this compact piece will be denoted $C(M)$. In the following, all modifications on
$M$ will be made outside of $C(M)$.

We need a topological property of $M$.

\begin{lem}\label{lem:acylind}
Let $M$ be a complete non-compact hyperbolic $3$-manifold with finite volume; $M$ is the
interior of some manifold $\barre M$. $\barre M$ has no incompressible annulus.
\end{lem}

\begin{proof}
Let $E_1,\dots,E_p$ be the ends of $M$ and $N_1,\dots,N_p$ the associated charts with
function $z_i$ on each end. Let $\Psi$ be a function as above with $\Psi''\le 0$ and
$\Psi'(2)=0$ and $\Psi'(t)>0$ for $t<2$. This implies that $g_\Psi$ has negative sectional
curvature on $\{1/2<z<2\}$ and $T(2)$ is totally geodesic. We endow each $E_i$ with this
new metric and we cut $\{z_i>2\}$ for each end. We get a manifold diffeomorphic to $\barre
M$ with a Riemannian metric with negative sectional curvature on the inside and totally
geodesic boundary. This defines a metric on $\barre M$

Let $A$ be an incompressible annulus in $\barre M$ endowed with the above metric; we can
deform it isotopically such that its boundary consists of geodesic circles in
$\partial\barre M$.
By
Theorem~6.12 in \cite{HaSc}, there is a minimal surface $S$ isotopic to $A$ with
the same boundary. Since the boundary of $\barre M$ is totally geodesic, $\partial S$ is
geodesic inside $S$. By the Gauss formula $K_S\le K_{\barre M}<0$. So the Gauss-Bonnet formula
$0=2\pi\chi_S=\int_S K_S<0$ gives a contradiction.
\end{proof}

Let $\barre M$ be a compact $3$-manifold whose boundary components are tori. Let $T$ be one
of these tori. By fixing a basis of the homology of $T$, we define a chart on $T\simeq
\S^1\times \S^1$ such that the basis of the homology is $(\S^1\times\{p\},
\{q\}\times\S^1)$. Let $\S^1\times D$ ($D$ the unit disk) be a solid torus, we then can
glue $\barre M$ and the solid torus by identifying the boundary using the chart on
$T$. The topology of this Dehn filling depends on the choice of the homology basis we
made. By making Dehn filling on each boundary component of $\barre M$, we get a closed
manifold $D(\barre M)$. One can easily see that, concerning the Heegaard genus, we have the
following inequality $g_H(D(\barre M))\le g_H(\barre M)$.

If $\barre M$ comes from a complete non-compact hyperbolic $3$-manifold $M$ with finite
volume, Rieck and Sedgwick \cite{RiSe} prove that the Dehn fillings can always be done (by
choosing particular homology basis) such that $g_H(D(\barre M))=g_H(\barre M)$ (the
acylindrical hypothesis in their theorem is satisfied because of Lemma~\ref{lem:acylind})
(see also Moriah and Rubinstein \cite{MoRu}).

We can now state our result concerning finite volume hyperbolic $3$-manifolds. 

\begin{thm}
Let $M$ be a complete non compact hyperbolic $3$-manifold with finite volume, we denote by
$\barre M$ the associated compact $3$-manifold with boundary. If $g_H(\barre M)\ge 7$,
then any closed orientable minimal surface in $M$ has area at least $2\pi$ and any closed
non orientable minimal surface has area at least $\pi$.
\end{thm}

The proof is based on ideas that appear in \cite{CoHaMaRo}.

\begin{proof}
Let $T_1,\dots,T_p$ be the boundary tori of $\barre M$, because of the above discussion,
there are bases of the homology of $T_1,\dots,T_p$ such that the associated Dehn filling
$D(\barre M)$ has the same Heegaard genus as $\barre M$.

We are going to construct some Riemannian metric on $D(\barre M)$ to estimate the areas of
minimal surfaces in $M$. Let $\Psi$ be a function on $[1/2,\infty)$ such that
\begin{itemize}
\item $\Psi(z)=z$ on $[1/2, 1]$,
\item $\Psi'(z)>0$,
\item $\lim_\infty \Psi=2$.
\end{itemize}

Let $C(M)$ be the compact part of $M$ that contains all compact minimal surfaces in $M$.
Now, for each end $E_i$, we can find a chart $N_{v_1^i,v_2^i}$ such that $E_i\cap
C(M)=\emptyset$, $\Lambda(E_i)<\Lambda_0$ where $\Lambda_0$ is given by
Proposition~\ref{prop:max} for $g_\Psi$ and $t_0=1/3$. We also assume that the curves
$t\mapsto (tv_1^i,1)$ and $t\mapsto (tv_2^i,1)$ in $T_i(1)$ give the homology basis of
$T_i$ that we have fixed above.

Let us fix some large $L>0$, we are going to change the metric on $\{L\le z_i\le L+1\}$ in
order to perform the Dehn filling. The tori $T_i(c)$ is parametrized by
$(u\frac{v_1^i}{2\pi}+v\frac{v_2^i}{2\pi})$ where $(u,v)\in\S^1\times \S^1$. Then the metric
$g_\Psi$ on $N_{v_1^i,v_2^i}$ can be written
$$
g_\Psi=\frac1{\Psi^2(z_i)}(a^2\dd u^2+2b\dd u\dd v+c^2\dd v^2+\dd z_i^2)
$$
for some $a,b,c\in \R$. Let $\eta$ be a smooth non increasing function on $[L,L+1]$ such
that $\eta(z)=1$ near $L$ and $\eta(z)=((L+1)-z)/a$ near $L+1$. We then change the metric
on $\{L\le z_i\le L+1\}$ by 
$$
\frac1{\Psi^2(z_i)}(\dd z_i^2+\eta^2(z_i)a^2\dd u^2+2\eta(z_i)b\dd u\dd v+c^2\dd v^2)
$$
This new metric is singular at $z_i=L+1$. Actually, it consists in cutting $\{z_i\ge L\}$
from the end $E_i$ and gluing a solid torus along $T(L)$. To see this, Let
$(r,\theta)\in[0,1]\times\S^1$ be the polar coordinates on the unit disk and $h$ be the
map $D\times \S^1\to\S^1\times\S^1\times[L,L+1]$ defined by
$(r,\theta,v)\mapsto(\theta,v,L+1-r)$. The induced metric by $h$ on $D\times\S^1$ near
$r=0$ is then
$$
\frac1{\Psi^2(L+1-r)}(\dd r^2+r^2\dd \theta^2+2\frac ba r\dd\theta\dd v+c^2\dd v^2)
$$
which is well defined on $D\times \S^1$. The map $h$ tells us that we have performed the
Dehn filling we want. We also notice that the tori $T_i(c)=\{z_i=c\}$ have positive mean
curvature with respect to $\partial_{z_i}$ for $L\le c<L+1$.

Once all Dehn fillings are done, we have constructed a metric on $D(\barre M)$ (it depends
on the parameter $L$ that we need to adjust). Let us study the area of minimal surfaces in
$D(\barre M)$ with that metric. Let $\Sigma$ be a minimal surface in $D(\barre M)$, first
it can stay outside of all the $\{z_i\ge 1\}$, these correspond to minimal surfaces living
in the original hyperbolic part of $D(\barre M)$ so in $M$. These surfaces are the ones
whose areas we wish to bound from below. Since the foliation $\{T_i(c)\}_{c\in [1,L+1)}$
is mean
convex with respect to $\partial_{z_i}$, there is no minimal surface inside an end $\{z_i\ge
1\}$. Proposition~\ref{prop:max} tells us that a minimal surface that intersects
$\{z_i\ge 1/2\}$ but does not reach $T_i(L)$ never enters into $\{z_i>1\}$. So it stays in the
original hyperbolic part. So a minimal surface that meets $\{z_i=1\}$ meets necessarily
$\{z_i=L\}$. Thus it meets all tori $T_i(c)$ for $1\le c\le L$. Since $\lim_\infty
\Psi=2$, for large $z_i$ the metric $g_\Psi$ is close to the Euclidean metric and then there is
some constant $k$ that does not depend on $L$ such that 
$$
|\Sigma\cap\{1\le z_i\le L\}|\ge kL.
$$
So if $L$ is chosen large, the area of $\Sigma$ is large. More precisely, we choose $L$
such that $kL\ge \boA_1(M)+1$.

Since $C(M)$ is isometrically contained in $D(\barre M)$, we have $\boA_1(M)\ge
\boA_1(D(\barre M))$. The above discussion implies that any minimal surface $\Sigma$ in
$D(\barre M)$ either is contained in $C(M)$ or has area $|\Sigma|\ge kL\ge \boA_1(M)+1$.
So $\boA_1(M)=\boA_1(D(\barre M)$. Moreover $\boA_1(D(\barre M)$ is realized by a minimal
surface contained in $C(M)$ where the metric is hyperbolic and where we can apply the same
reasoning as in the proof of Theorem~C and using the fact that $g_H(D(\barre
M))=g_H(\barre M)\ge 7$.
\end{proof}

In a finite volume hyperbolic $3$-manifold, it is also interesting to find a good lower
bound of the area of non compact minimal surfaces. We notice that in this case the
estimates $\pi|\chi(\Sigma)|\le |\Sigma|\le 2\pi|\chi(\Sigma)|$ are still valid for
properly embedded stable minimal surfaces (see \cite{CoHaMaRo}).


\bibliographystyle{amsplain}
\bibliography{../reference.bib}

\providecommand{\bysame}{\leavevmode\hbox to3em{\hrulefill}\thinspace}
\providecommand{\MR}{\relax\ifhmode\unskip\space\fi MR }
\providecommand{\MRhref}[2]{%
  \href{http://www.ams.org/mathscinet-getitem?mr=#1}{#2}
}
\providecommand{\href}[2]{#2}
\begin{thebibliography}{10}

\bibitem{Alm}
Frederick~Justin Almgren, Jr., \emph{The homotopy groups of the integral cycle
  groups}, Topology \textbf{1} (1962), 257--299.

\bibitem{CaGo}
A.~J. Casson and C.~McA. Gordon, \emph{Reducing {H}eegaard splittings},
  Topology Appl. \textbf{27} (1987), 275--283.

\bibitem{CoDeL}
Tobias~H. Colding and Camillo De~Lellis, \emph{The min-max construction of
  minimal surfaces}, Surveys in differential geometry, {V}ol.\ {VIII}
  ({B}oston, {MA}, 2002), Surv. Differ. Geom., VIII, Int. Press, Somerville,
  MA, 2003, pp.~75--107.

\bibitem{CoHaMaRo}
Pascal Collin, Laurent Hauswirth, Laurent Mazet, and Harold Rosenberg,
  \emph{Minimal surfaces in finite volume non compact hyperbolic
  $3$-manifolds}, preprint, arXiv:1405.1324.

\bibitem{DeLPe}
Camillo De~Lellis and Filippo Pellandini, \emph{Genus bounds for minimal
  surfaces arising from min-max constructions}, J. Reine Angew. Math.
  \textbf{644} (2010), 47--99.

\bibitem{DeLTa}
Camillo De~Lellis and Dominik Tasnady, \emph{The existence of embedded minimal
  hypersurfaces}, J. Differential Geom. \textbf{95} (2013), no.~3, 355--388.

\bibitem{Fed}
Herbert Federer, \emph{Geometric measure theory}, Die Grundlehren der
  mathematischen Wissenschaften, Band 153, Springer-Verlag New York Inc., New
  York, 1969.

\bibitem{Gal}
Jos{\'e}~A. G{\'a}lvez, \emph{Surfaces of constant curvature in 3-dimensional
  space forms}, Mat. Contemp. \textbf{37} (2009), 1--42.

\bibitem{Has}
Joel Hass, \emph{Acylindrical surfaces in {$3$}-manifolds}, Michigan Math. J.
  \textbf{42} (1995), 357--365.

\bibitem{HaSc}
Joel Hass and Peter Scott, \emph{The existence of least area surfaces in
  {$3$}-manifolds}, Trans. Amer. Math. Soc. \textbf{310} (1988), 87--114.

\bibitem{Ket}
Daniel Ketover, \emph{Degeneration of min-max sequences in three-manifold},
  preprint, arXiv:1312.2666.

\bibitem{LiYa}
Peter Li and Shing~Tung Yau, \emph{A new conformal invariant and its
  applications to the {W}illmore conjecture and the first eigenvalue of compact
  surfaces}, Invent. Math. \textbf{69} (1982), 269--291.

\bibitem{MaNe3}
Fernando~C. Marques and Andr{\'e} Neves, \emph{Rigidity of min-max minimal
  spheres in three-manifolds}, Duke Math. J. \textbf{161} (2012), 2725--2752.

\bibitem{MaNe}
\bysame, \emph{Min-max theory and the {W}illmore conjecture}, Ann. of Math. (2)
  \textbf{179} (2014), 683--782.

\bibitem{MeSiYa}
William Meeks, III, Leon Simon, and Shing~Tung Yau, \emph{Embedded minimal
  surfaces, exotic spheres, and manifolds with positive {R}icci curvature},
  Ann. of Math. (2) \textbf{116} (1982), 621--659.

\bibitem{Mor3}
Frank Morgan, \emph{Examples of unoriented area-minimizing surfaces}, Trans.
  Amer. Math. Soc. \textbf{283} (1984), 225--237.

\bibitem{Mor2}
\bysame, \emph{A regularity theorem for minimizing hypersurfaces modulo
  {$\nu$}}, Trans. Amer. Math. Soc. \textbf{297} (1986), 243--253.

\bibitem{MoRu}
Yoav Moriah and Hyam Rubinstein, \emph{Heegaard structures of negatively curved
  {$3$}-manifolds}, Comm. Anal. Geom. \textbf{5} (1997), 375--412.

\bibitem{NaSo}
Hossein Namazi and Juan Souto, \emph{Heegaard splittings and pseudo-{A}nosov
  maps}, Geom. Funct. Anal. \textbf{19} (2009), 1195--1228.

\bibitem{Pit2}
Jon~T. Pitts, \emph{Existence and regularity of minimal surfaces on
  {R}iemannian manifolds}, Mathematical Notes, vol.~27, Princeton University
  Press, Princeton, N.J.; University of Tokyo Press, Tokyo, 1981.

\bibitem{RiSe}
Yo'av Rieck and Eric Sedgwick, \emph{Persistence of {H}eegaard structures under
  {D}ehn filling}, Topology Appl. \textbf{109} (2001), 41--53.

\bibitem{RiRo}
Manuel Ritor{\'e} and Antonio Ros, \emph{Stable constant mean curvature tori
  and the isoperimetric problem in three space forms}, Comment. Math. Helv.
  \textbf{67} (1992), 293--305.

\bibitem{ScSi}
Richard Schoen and Leon Simon, \emph{Regularity of stable minimal
  hypersurfaces}, Comm. Pure Appl. Math. \textbf{34} (1981), 741--797.

\bibitem{Sha}
Peter~B. Shalen, \emph{Hyperbolic volume, {H}eegaard genus and ranks of
  groups}, Workshop on {H}eegaard {S}plittings, Geom. Topol. Monogr., vol.~12,
  Geom. Topol. Publ., Coventry, 2007, pp.~335--349.

\bibitem{Sim}
Leon Simon, \emph{Lectures on geometric measure theory}, Proceedings of the
  Centre for Mathematical Analysis, Australian National University, vol.~3,
  Australian National University Centre for Mathematical Analysis, 1983.

\bibitem{Smi}
F.~Smith, \emph{On the existence of embedded minimal {$2$}-spheres in a
  {$3$}-sphere, endowed with an arbitrary {R}iemannian metric}, Ph.D. thesis,
  University of {M}elbourne, 1982, Supervisor : {L}eon {S}imon.

\bibitem{Sout}
Juan Souto, \emph{Geometry, {H}eegaard splittings and rank of the fundamental
  group of hyperbolic {$3$}-manifolds}, Workshop on {H}eegaard {S}plittings,
  Geom. Topol. Monogr., vol.~12, Geom. Topol. Publ., Coventry, 2007,
  pp.~351--399.

\bibitem{Zho}
Xin Zhou, \emph{Min-max minimal surface in $({M}^{n+1},g)$ with ${R}ic_g>0$ and
  $2\le n\le 6$}, preprint, arXiv:1210.2112.

\end{thebibliography}

\end{document}